\documentclass{article}

 \usepackage[final]{neurips_2021}
\usepackage{filecontents}
\usepackage{bbm}
\usepackage{fullpage}
\usepackage{xr}
\usepackage[pdftex]{graphicx}
\usepackage[numbers]{natbib}
\usepackage{hyperref}       
\usepackage{url,color}            
\usepackage{booktabs}       
\usepackage{amsfonts}       
\usepackage[ruled,vlined]{algorithm2e}
\SetKwInput{KwInput}{Input}                
\SetKwInput{KwOutput}{Output} 
\usepackage{nicefrac}       
\usepackage{microtype}      
\usepackage{comment,subcaption}
\usepackage{amsmath,amsthm,enumitem,amssymb}
\usepackage{titlesec}
\usepackage{algorithm2e}
\usepackage{xcolor}  

\title{Bootstrapping the error of Oja's algorithm}

\author{Robert Lunde  \\ University of Michigan \\ \texttt{rlunde@umich.edu} 
        \and 
        \textbf{Purnamrita Sarkar}
        \\  University of Texas at Austin \\ \texttt{purna.sarkar@austin.utexas.edu}   
        \and 
        \\ 
        \textbf{Rachel Ward} \\ University of Texas at Austin \\ \texttt{rward@math.utexas.edu}}

\newcommand{\hd}{Hoeffding decomposition\xspace}

\newcommand{\bk}{\color{black}}
\newcommand{\trace}{\mathrm{trace}}
\newtheorem{remark}{Remark}
\newtheorem{theorem}{Theorem}
\newtheorem{lemma}{Lemma}
\newtheorem{proposition}{Proposition}
\newtheorem{corollary}{Corollary}

\newcommand{\bb}[1]{\left(#1\right)}

\newcommand{\bas}[1]{\begin{align*}#1\end{align*}}
\newcommand{\ba}[1]{\begin{align}#1\end{align}}

\newcommand{\E}{\mathbb{E}}

\newcommand{\var}{\mathrm{Var}}
\newcommand{\oja}{\hat{v}_1}
\newcommand{\norm}[1]{\left\lVert#1\right\rVert}

\newcommand{\vorth}{V_\perp}

\newcommand{\bM}{\mathbb{M}}
\newcommand{\vp}{\widehat{V}_\perp}

\newtheorem*{theorem*}{Theorem}

\newcommand{\bigvar}{\bar{\mathbb{V}}_n}

\newcommand{\sgn}{\text{sgn}}

\begin{document}

\maketitle

\begin{abstract}
 We consider the problem of quantifying uncertainty for the estimation error of the leading eigenvector from Oja's algorithm for streaming principal component analysis, where the data are generated IID from some unknown distribution.  By combining classical tools from the U-statistics literature with recent results on high-dimensional central limit theorems for quadratic forms of random vectors and concentration of matrix products, we establish a weighted $\chi^2$ approximation result for the $\sin^2$ error between the population eigenvector and the output of Oja's algorithm.
Since estimating the covariance matrix associated with the approximating distribution requires knowledge of unknown model parameters, we propose a multiplier bootstrap algorithm that may be updated in an online manner.  We establish conditions under which the bootstrap distribution is close to the corresponding sampling distribution with high probability, thereby establishing the bootstrap as a consistent inferential method in an appropriate asymptotic regime.       
\end{abstract}
\section{Introduction}
Since its discovery over a century ago \cite{pearson-pca}, principal component analysis (PCA) has been a cornerstone of data analysis.  In many applications, dimension reduction is paramount and PCA offers an optimal low-rank approximation of the original data.  PCA is also highly interpretable as it projects the dataset onto the directions that capture the most variance known as principal components. 

Important applications of PCA include image and document analysis, where the largest few principal components may be used to compress a large dimensional dataset to a manageable size without incurring much loss; for a discussion of some other applications of PCA, see for example, \cite{joliffe-pca}. In these settings, the original dimensionality, which could be the number of pixels in an image or the vocabulary size after removing stop-words,  is in the tens of thousands. An offline computation of the principal components would require the computation of eigenvectors of the sample covariance matrix. However, in high-dimensional settings, storing the covariance matrix and subsequent eigen-analysis can be challenging. Streaming PCA methods have gained significant traction owing to their ability to iteratively update the principal components by considering one data-point at a time.


\
One of the most widely used algorithms for streaming PCA is Oja's algorithm, proposed in the seminal work of~\cite{oja-simplified-neuron-model-1982}. 
   Oja's algorithm involves the following update rule:
\ba{\label{eq:oja}
w_{t+1}-w_t= \eta (w_t^T X_t) X_t; \qquad w_{t+1}^T w_{t+1}=1,
}
where $X_t\in \mathbb{R}^d$ is the $t^{th}$ data point and $w_t$ is the current estimate for the leading eigenvector of $\Sigma = \mathbb{E} X X^T$  after $t$ data-points have been seen.
 The parameter $\eta$ can be thought of as a learning rate, which can either be fixed or varied as a function of $t$. In this paper we fix the learning rate, similar to~\cite{jain2009}.

\noindent {\bf Contribution:}
 In the present work, we consider the problem of uncertainty quantification for the estimation error of the leading eigenvector from Oja's algorithm, which is one of the most commonly used streaming PCA algorithms. Our contributions may be summarized as follows:
 \begin{enumerate}
     \item  We derive a high-dimensional weighted $\chi^2$ approximation to the $\sin^2$ error for the leading eigenvector of Oja's algorithm.  We recover the optimal convergence rate $O(1/n)$ while allowing $d$ to grow at a sub-exponential rate under suitable structural assumptions on the covariance matrix, matching state-of-the-art theoretical results for consistency of Oja's algorithm.  Our result provides a distributional characterization of the $\sin^2$ error for Oja's algorithm for the first time in the literature.  The approximation holds for a wide range of step sizes.           
     \item Since the weighted $\chi^2$ approximation depends on unknown parameters, we propose an online bootstrap algorithm and establish conditions under which the bootstrap is consistent.  Our bootstrap procedure allows the approximation of important quantities such as the quantiles of the error associated with Oja's algorithm for the first time.

 \end{enumerate}
\noindent {\bf Prior analysis of Oja's algorithm.} While Oja's algorithm was invented in 1982 
it was not until recently that the theoretical workings of Oja's algorithm have been understood.  A number of papers in recent years have focused on proving guarantees of convergence of the iterative update in \eqref{eq:oja} toward the principal eigenvector of the (unknown) covariance matrix  $\E XX^T$, which can be recast as stochastic gradient descent (SGD) on the quadratic objective function
\ba{\label{eq:ojaobj}
	\min_{\substack{w\\w^Tw=1}} -\trace(w^T \Sigma w), \quad \quad \Sigma=\E XX^T,
}
projected onto the non-convex unit sphere.  We assume that the data-points are mean zero.
Despite being non-convex and thus falling outside the framework for which theory for stochastic gradient descent convergence is firmly established, the output of Oja's algorithm be viewed as a product of random matrices and shares similar structure to other important classes of non-convex problems, such as matrix completion~\cite{jain2013matcompl,keshavan2010completion}, matrix sensing~\cite{jain2013matcompl}, and subspace tracking~\cite{balzano2010sstracking}.
Thus, studying this optimization problem serves as a natural first step toward understanding the behavior of SGD in more general non-convex settings.  

%

Let $v_1$ denote the principal eigenvector of $\Sigma$, and let $\oja = w_n$ be the solution to the stochastic iterative method applying Eq~\ref{eq:oja}. Finally, let $\lambda_1 > \lambda_2$ be the first and second principal eigenvalues of $\Sigma$.   \bk Sharp rates of convergence for Oja's updates were established in ~\cite{jain2016streaming}. Under boundedness assumptions on $\|X_iX_i^T-\Sigma\|$, they show that with constant probability, the square of the sine of the angle between $v_1$ and $w$ satisfies:
\ba{\label{eq:jainoja}
1-(v_1^T \oja)^2=O\bb{\frac{1}{n}}
}
where the $O$ hides a constant which depends in the optimal way on the eigengap between the top two eigenvalues, and independent of $n$ or $d$, improving on previous error bounds for Oja's algorithm ~\cite{desa2014alecton,hardt2014noisy,bal2013pcastreaming,shamir2016pca,mitliagkas2013pca,pmlr-v49-balcan16a} which showed convergence rates that deteriorate with the ambient dimension $d$, and thus did not fully explain the efficiency of Oja's update. 
This sharp rate is remarkable, as it matches the error of the principal eigenvector of the sample covariance matrix, which is the batch or offline version of PCA. 
Other notable work include~\cite{Liu2018oja,liang2021optimality} for unbounded $X_i$, analysis of Oja's algorithm for computing top $k$ principal components~\cite{ALOW2017,huang2021streaming}.

 

\noindent {\bf The bootstrap.}
The bootstrap, proposed by \cite{efron1986},  
is one of the most widely used methods for uncertainty quantification in machine learning and statistics and accordingly has a vast literature. We refer the reader to \cite{hall-edgeworth, vandervaart-asymptotic}  for expositions on the classical theory of the bootstrap for IID data. Recently, since the groundbreaking work of \cite{cck2013,chernozhukov2017}, the bootstrap has seen a renewed surge of interest in the context of high-dimensional data where $d$ can be  potentially exponentially larger than $n$.  Of particular relevance to the present work are high-dimensional central limit theorems (CLTs) for quadratic forms, which have been studied by \cite{10.1214/15-EJS1090, 10.1093/biomet/asz020, giessing2020bootstrapping}.  In particular, our CLT for the estimation error of Oja's algorithm invokes a modest adaptation of \cite{10.1093/biomet/asz020} to independent but non-identically distributed random variables.    
%
In machine learning, bootstrap methods have been used to estimate the uncertainty of randomized algorithms such as bagging and random forests~\cite{lopes2019bagging}, sketching for large scale singular value decomposition (SVD)~\cite{lopes2020svd}, randomized matrix multiplication~\cite{lopes2019matmult}, and randomized least squares~\cite{pmlr-v80-lopes18a}.    

A standard notion of bootstrap consistency is that, conditioned on the data, the distribution of the suitably centered and scaled bootstrap functional approaches the true distribution with high probability in some norm on probability measures, typically the Kolmogorov distance, which is the supremum of the absolute pointwise difference between two CDFs. 
Bootstrap consistency is often established by deriving a Gaussian approximation for the sampling distribution and showing that the bootstrap distribution is close to the corresponding Gaussian approximation with high probability.

It may seem that if one knows that the approximating distribution of a statistic is Gaussian, this defeats the purpose of bootstrap. However, \emph{for most statistics, the parameters of the normal approximation depend on unknown model parameters, and have to be estimated if one intends to use the normal approximation}. Furthermore, the CLT only gives a \emph{first-order} correct approximation of the target distribution, i.e. with $O(1/\sqrt{n})$ error. In contrast, the bootstrap of a suitably centered and scaled statistic has been shown to be higher order correct for many functionals~\cite{gotze1996,hall-edgeworth,helmers1991}. 


\noindent {\bf Quantifying uncertainty for SGD.}
Behind the recent success of neural networks in a wide range of sub-fields of machine learning, the workhorse algorithm has become Stochastic gradient descent (SGD)~\cite{polyak1992sgd,nemirovsky2009,robbins1951}.  
For establishing consistency of bootstrap, one requires to establish asymptotic normality~\cite{fabian1968,polyak1992sgd,Ruppert1988EfficientEF,bach2011sgd}. There has also been many works on uncertainty estimation of SGD~\cite{chen2020SGD,Li2017SGDaistats,fang2018boot,su2018uncertainty}. 
However, \emph{all these works are for convex, and predominantly strongly convex loss functions.}  Only recently, ~\cite{yu2020nonconvexAnalysis} has established asymptotic normality for nonconvex loss functions under dissipativity conditions and appropriate growth conditions on the gradient, which are weaker conditions than strong convexity but not significantly so. 

Now, in Section~\ref{sec:notation} we present notation and do setup, present our main theoretical results in Section~\ref{sec:main}, followed by simulations in Section~\ref{sec:exp}.

\section{Preliminaries}\label{sec:notation}
We consider a row-wise IID triangular array, where the random vectors $\{X_i\}$ in the $n^{th}$ row take values in $\mathbb{R}^{d_n}$, with $\mathbb{E}[X_i] = 0$ and $\mathrm{Var}(X_i) = \Sigma_{n}$. Note that the triangular array allows $\{X_1,\dots,X_n\}$ to come from a different distribution for each $n$ and the setting where $d$ is fixed and $n$ grows is a special case. For readability, we drop the subscript $n$ from $\Sigma_n$.  We use $\| \cdot \|$ to denote the Euclidean norm for vectors and the operator norm for matrices and $\| \cdot \|_F$ to denote the Frobenius norm.   
\bk


Expanding out the recursive definition in Eq~\ref{eq:oja}, we see that Oja's iteration can be expressed as $w_{t+1}=(I_d+\eta X_tX_t^T) w_t$. Thus, after $n$ iterations the vector can be written as a matrix-vector product, where the matrix is a product of $n$ independent matrices. 
Expanding out the recursive definition, we get:
\ba{\label{eq:ojamatprod}
	B_n:=\prod_{i=1}^n (I_d+\eta  X_iX_i^T)\qquad \oja=\frac{B_nu_0}{\|B_n u_0\|},
}
where $I_d$ is a $d\times d$ identity matrix.
where $u_0$ is a random unit vector in $d$ dimensions. In the scalar case, when $\eta=1/n$, for large $n$, the numerator of Eq~\ref{eq:ojamatprod} behaves like $\exp(\sum_i X_i^2/n)$, which in turn converges to $\exp(E[X_1^2])$.  For matrices, one hopes that, by independence, a result of the same flavor will hold. 
And in fact if it does hold, then for $\eta=\frac{\log n}{n}$, the numerator in Eq~\ref{eq:ojamatprod} will concentrate around $\exp(\log n \Sigma)$. The spectrum of this matrix is dominated by the principal eigenvector, i.e. the ratio of the first eigenvalue to the second one is $\exp(\log n(\lambda_1-\lambda_2))$, where $\lambda_i$ is the $i^{th}$ eigenvalue of the covariance matrix $\Sigma$. This makes it clear that \emph{Oja's algorithm is essentially a matrix vector product of this matrix exponential (suitably scaled) and a random unit vector.}

However, the intuition from the scalar case  is nontrivial to generalize to matrices due to non-commutativity. Limits of products of random matrices have been studied in mathematics in the context of ergodic theory on Markov chains (see~\cite{FURMAN2002931,ledrappier2001asymptotic,randwalk-groups,emme2017matprod} etc.). However, until recent results of \cite{huang2020matrix}, which extended and improved results in  \cite{henriksen2019matprod}, there has not been much work on quantifying the exact \emph{rate of convergence}, or finite-sample large deviation bounds for how a random matrix product deviates from its expectation.

We reparametrize $\eta$ as $\eta_n/n$, where $\eta_n$ is chosen carefully to obtain a suitable error rate. Note that this is not a scheme where we decrease $\eta$ over time as in~\cite{henriksen2019adaoja}, but hold it as a constant which is a function of the total number of data-points. \bk


\subsection{The Hoeffding decomposition}
The Hoeffding decomposition, attributed to \cite{hoeffding1948}, is a key technical tool for studying the asymptotic properties of U-statistics. However, the idea generalizes far beyond U-statistics; see Supplement Section \ref{sec:HD} for further discussion.  In the present work, we use Hoeffding decompositions for matrix and vector-valued functions of independent random variables taking values in $\mathbb{R}^d$ to facilitate analysis for $B_n$. 

A concept closely related to the Hoeffding decomposition is the more well-known H\'{a}jek projection,  which gives the best approximation (in an $L_2$ sense) of a general function of $n$ independent random variables by a function of the form $\sum_i g_i(X_i)$, where $g_i$ are measurable functions satisfying a square integrability condition. The H\'{a}jek projection facilitates distributional approximations for complicated statistics since this linear projection is typically more amenable to analysis. However, establishing a central limit theorem requires showing the negligibility of a remainder term, which can be large if the projection is not accurate enough.     

The H\'{a}jek projection may be viewed as the first-order term in the Hoeffding decomposition, a general way of representing functions of independent random variables.  The Hoeffding decomposition consists of a sum of projections onto a linear space, quadratic space, cubic space, and so on. Each new space is chosen to be orthogonal to the previous space. Thus, the \hd can be thought of as a sum of terms of increasing levels of complexity.  Even if the remainder of the H\'{a}jek projection turns out to be small, the Hoeffding decomposition can be easier to work with due to the orthogonality of the projections.

\bigskip

\noindent {\bf The Hoeffding decomposition for the matrix product.}
Let  $Y_i = X_iX_i^T- \Sigma$ and let $S \subseteq \{1, \ldots n\}$.  By Corollary \ref{cor-hoeff-decomp-oja} of the Supplement Section \ref{sec:HD},  the Hoeffding Decomposition for $B_n$ is given by:
\begin{align}
\label{eq:hayek-proj-decomp}
B_n = \sum_{k=0}^n T_k,\qquad T_k =  
\sum_{|S|=k} H^{(S)}.
\end{align}
%
%
%
where  $H^{(S)} = \prod_{i=1}^n A_i^{(S)}$ and  $ A_i^{(S)}$ is given by:   
$\begin{aligned}
A_{i}^{(S)} &= \begin{cases}
\frac{\eta_n}{n} Y_i & \text{if} \ i \in S  \\ 
I+\frac{\eta_n}{n} \Sigma & \text{otherwise}
\end{cases}
\end{aligned}$.
%
%
 
 The above expansion has favorable properties that facilitate second-moment calculations.  In fact, as a consequence of the orthogonality property of Hoeffding projections, we have that
 \begin{align*}
 \mathbb{E}\left[\norm{B_n}_F^2 \right] &= \sum_{k=0}^n \sum_{|S| = k } \mathbb{E}\left[\|\prod_{i=1}^n H_i^{(S)}\|_F^2\right]\\  \mathbb{E}\left[ \|B_n x\|^2 \right] &= \sum_{k=0}^n \sum_{|S| = k} \mathbb{E}\left[\|\prod_{i=1}^n H_i^{(S)}x\|^2\right]  
 \end{align*}
 where the second statement holds for any $x \in \mathbb{R}^d$; see Proposition \ref{prop:orth-hoeffding-proj} in Supplement Section \ref{sec:HD}. \\  
%
%

%

%
\subsection{Online bootstrap for streaming PCA}


To approximate the sampling distribution, we consider a Gaussian multiplier bootstrap procedure. As observed by \cite{cck2013}, a Gaussian multiplier random variable eliminates the need to establish a Gaussian approximation for the bootstrap since conditional on the data, it is already Gaussian. It is not hard to see that this is a natural candidate for the online setting; the multiplier bootstrap has been used for bootstrapping the stochastic gradient descent estimator in ~\cite{fang2018boot}.

We present our bootstrap in Algorithm~\ref{alg:boot}. In our procedure, we update $m+1$ vectors at every iteration. The first one is $\hat{v}$, which will result in the final Oja estimate of the first principal component. The other vectors $\{v^{*(j)},j=1,\dots m\}$ are obtained by perturbing the basic Oja update (Eq~\ref{eq:oja}).

	
\begin{algorithm}
	\SetAlgoLined
	\KwInput{Datapoints $X_1,\dots, X_n$, stepsize $\eta$, number of bootstrap replicates $m$}     
	\KwOutput{Oja's solution $\oja$ and $m$ bootstrapped versions of it $v_1^{*(1)},\dots,v_1^{*(m)} $}
	Draw $g\sim N(0,I_d)$\\
	Create unit vector $u_0\leftarrow g/\|g\|$\\
	Initialize $\oja,v_1^{*(1)},\dots,v_1^{*(m)}\leftarrow u_0$ \\
	
	\For{t=2,\dots, n}{
		Update $\oja\leftarrow \oja +\eta (X_t^T \oja)\oja$\\
		Normalize $\oja$ to have unit norm;\\
		\For{i=1:m}{
			Draw $W_i\sim N(0,1/2)$;\\
			Let $h^{(i)} \leftarrow (X_t^T v_1^{*(i)})X_{t}$;\\
			Let $g^{(i)} \leftarrow (X_{t-1}^T v_1^{*(i)})X_{t-1}$;\\
			Update $v_1^{*(i)}\leftarrow	v_1^{*(i)}+\eta \bb{h^{(i)}+W_i (h^{(i)}-g^{(i)})} $;\\
			Normalize $v_1^{*(i)}$ to have unit norm;
	}
	}
	\caption{\label{alg:boot}Bootstrap for Oja's algorithm}
\end{algorithm}

The $W_i$'s are the multiplier random variables, which are scaled mean zero scaled Gaussians with variance $1/2$. The update of the $v^{*(j)}$ is novel because it preserves the mean and the variance of the original Oja estimator while not requiring access to the full sample covariance matrix. Consequently, we can make our updates online and attain both a point estimate and a confidence interval for the principal eigenvector, \emph{while increasing the computation and storage by only a factor of $m$.} 

\section{Main results}\label{sec:main}
In this section we present our main contributions: a CLT for the error of Oja's algorithm and consistency of an online multiplier bootstrap for error.  
\subsection{Central limit theorem for the error of Oja's algorithm}
 We start by stating a CLT for the error of Oja's algorithm.  To state this theorem, we will need to introduce some notation. 

Let  $\hat{v}_1$ denote the Oja vector and $V_\perp$ the $d\times d-1$ matrix with $2,\ldots,d$ eigenvectors of $\Sigma$ on its columns.  Note that $V_\perp$ is not uniquely defined, but $V_\perp V_\perp^T = I- v_1 v_1^T$ is if the leading eigenvalue is distinct and consequently, norms of the form $\|V_\perp^T x \|$ for $x \in \mathbb{R}^d$ are well-defined. Let $\lambda_1\geq \dots \geq \lambda_d$ denote the eigenvalues of $\Sigma$ and  $\Lambda_\perp$ be a diagonal matrix with $\Lambda_\perp(i,i)=(1+\eta_n\lambda_{i+1}/n)/(1+\eta_n\lambda_1/n)$, $i=1,\dots,d-1$.  Also let
 \ba{
 \mathbb{M} :=  \mathbb{E}\left[V_\perp^T(X_1^Tv_1)^2 X_1X_1^T V_\perp \right ]
 }
Now we define 
\begin{align}\label{eq:bigvar}
    \bigvar &= \frac{\eta_n}{n}\sum_i \E [\vorth \Lambda_{\perp}^{i-1}\vorth^T(X_i X_i^T-\Sigma)v_1v_1^T(X_i X_i^T-\Sigma)\vorth \Lambda_\perp^{i-1}\vorth^T]\notag\\
   &= \frac{\eta_n}{n}V_\perp \left(\sum_i \Lambda_\perp^{i-1}\bM\Lambda_\perp^{i-1}\right)V_\perp^T
\end{align}


We have the following result: 

\begin{theorem}
\label{thm-oja-clt}
Suppose that $u_0$ is drawn from the uniform distribution on $\mathcal{S}^{d-1}$, $\lambda_1 = O(1)$.  
Choose $\eta_n\rightarrow\infty$ such that
$nd  \cdot \exp(-\eta_n(\lambda_1-\lambda_2))\rightarrow 0$, 
$\frac{ (\eta_n \vee \ \log d) \ \eta_n^2  (M_d^2\vee 1)}{n} \rightarrow 0$, where $M_d = \mathbb{E}[\norm{X_iX_i^T - \Sigma}^2]$.  Further, let $\widetilde{Z}_n$ be a mean $0$ Gaussian matrix such that $\text{Var}(\widetilde{Z}_n) = \text{Var}( (X_1X_1^T - \Sigma)v_1 )$ and suppose that:
\begin{align}\label{eq:clt-lb}
\norm{\mathbb{M}}_F &\geq c >0
\end{align}
\begin{align}\label{eq:clt-norm-bound}
\frac{ \mathbb{E}\left[\bigl\|V_\perp^T\widetilde{Z}_n \bigr\|^6\right] \vee \mathbb{E}\left[\norm{V_\perp^T(X_1 X_1^T - \Sigma)v_1}^6\right]}{\norm{\mathbb{M}}_F^3} = o(n) 
\end{align} 
Then, for a sequence of Gaussian distributions $\{Z_n\}_{n \geq 1}$ with mean $0$ and covariance matrix $\bigvar$ (see Eq~\ref{eq:bigvar}), the following holds:
\begin{align}
\label{eq:kolmogorov-consistency}
\sup_{t \in \mathbb{R}}\left|P\left( n/\eta_n \ \sin^2(\hat{v}_1, v_1) \leq t\right) - P(Z_n^T Z_n \leq t)\right| \rightarrow 0  
\end{align}
\end{theorem}

Theorem~\ref{thm-oja-clt} is very general. We allow the dimension to grow with the number of observations, which is typical in the high-dimensional bootstrap literature. Note that the case of fixed $d$ and growing $n$ is also a special case of this setup.

We want to point out that while previous literature obtained sharp bounds on the $\sin^2$ error $1-(v_1^T\hat{v}_1)^2$, we go a step further. \emph{We establish an approximating distribution for $n/\eta_n (1-(v_1^T\hat{v}_1)^2)$. }

\begin{remark}[Condition on norm]
For simplicity, we assume $\lambda_1=O(1)$, which can be easily relaxed to grow slowly with $n$. 
We do not assume that the $\|X_iX_i^T-\Sigma\|_2$ is bounded almost surely. However, the 
norm of $X_iX_i^T-\Sigma$ comes into play implicitly via the assumption in Eq~\ref{eq:clt-norm-bound}. 
Consider the case where $X_i$ are drawn from some multivariate Gaussian distribution. We use this to build intuition about the assumptions in Eq~\ref{eq:clt-lb} and~\ref{eq:clt-norm-bound}. In this case, $X_1^T\vorth$ is a Gaussian of independent entries and thus $\mathbb{E}\left[\norm{V_\perp(X_1 X_1^T - \Sigma)v_1}^6\right] = \E\|X_1^Tv_1\|^6\E \bb{\sum_{j>1} (X_j^T v_j)^2}^3$. Note that $\sum_{j>1} ((X_j^T v_j)^2-\lambda_j)$ is a sub-exponential random variable with parameters $(c_1\sum_{j>1} \lambda_j, c_2)$. Furthermore, $\|\bM\|_F^2=\lambda_1\sum_{i>1}\lambda_i$. Thus Eq~\ref{eq:clt-norm-bound}  reduces to checking if
\bas{
\frac{\lambda_1^{3/2}(\sum_{j>1}\lambda_j)^3}{\left(\sum_i \lambda_i^2\right)^{3/2}} = o(n)
}
\end{remark}

\begin{remark}[Coordinates with summable sub-Gaussian parameters]
\label{remark:summable-subG}
Eq \ref{eq:clt-norm-bound} imposes a growth condition on the moments of both the data and a Gaussian analog.  One setting for which both growth rates are in fact bounded is if the coordinates of $X$ are sub-Gaussian and the sub-Gaussian parameters satisfy $\sum_{i=1}^d \nu_i < C < \infty$ following similar arguments to Proposition \ref{prop:vardecay}.

\end{remark}

\begin{remark}[Constant vs Adaptive Learning Rate]

Adaptive learning rates are also commonly studied in the literature on Oja's algorithm and have the advantage that they require no prior knowledge of the sample size.  It should be noted that our results hold for a wide range of learning rates, ranging from $\log(nd) \ll \eta_n \ll n^{1/3}$, so our results will still apply so long as in the initial guess of the sample size is not off by orders of magnitude. 
We leave a detailed study of the adaptive learning rate setting to future work.

\end{remark}

As a corollary of our main theorem, we obtain the following error bound on the $\sin^2$ error.
\begin{corollary}\label{cor:sinsq} Under the conditions in Theorem \ref{thm-oja-clt}, we have
\begin{align*}
 \sin^2(\oja,v_1) =  O_P\bb{\frac{\eta_n M_d}{n(\lambda_1-\lambda_2)}}
\end{align*}
\end{corollary}
\begin{remark}[Comparison with previous work]

As a byproduct of our analysis, we recover the sharpest convergence rates for Oja's algorithm in the literature. If we set $\eta_n=c_1\log n  d /(\lambda_1-\lambda_2)$, for large enough $c_1$, the dominating term in the error is $O_P\bb{\dfrac{M_d\log n d \bk }{n(\lambda_1-\lambda_2)^2} }$ under mild conditions on $d$. This matches the bound in~\cite{jain2016streaming}.
\end{remark}

\begin{remark}[Rate of convergence in Kolmogorov distance]
To simplify the theorem statement, we have stated Theorem \ref{thm-oja-clt} without giving an explicit rate of convergence in the Kolmogorov distance.  Convergence rates depend on the rate of decay of the remainder terms, which are worked out in Supplement Section \ref{sec:CLT-support}, and the magnitude of the quantity in Eq \ref{eq:clt-norm-bound}.  The contribution of the latter quantity to the rate is worked out in the IID case in \cite{10.1093/biomet/asz020}. 
\end{remark}

\begin{remark}[Lower bound on norm]
While our rate matches the sharp bounds in literature and our assumptions on norm upper bounds are similar or weaker than previous work, we do assume a lower bound on the Frobenius norm of the covariance matrix as in Eq~\ref{eq:clt-lb}. Note that if indeed all $X_i$'s were a scalar multiple of $v_1$, then the $\bigvar$ matrix in Eq~\ref{eq:bigvar} will be zero. This will lead to a perfect point estimate, but there will not be any variability from the data and hence there will be no non-degenerate approximation.
The lower bound on the norm is not resulting from loose analysis. Similar lower bounds on the variance are imposed in the high-dimensional CLT literature \cite{cck2013,chernozhukov2017}. 
\end{remark}
Now we provide a proof sketch of Theorem~\ref{thm-oja-clt} below.
\begin{proof}[Proof sketch for Theorem~\ref{thm-oja-clt}]
We provide the main steps in our derivation. The detailed calculations can be found in Supplement Section~\ref{sec:B}.
\begin{enumerate}[label={\arabic{enumi}}.,ref={Step \arabic{enumi}},leftmargin=*]
\item 
We start by expressing the $\sin^2$ error as a quadratic form:
\begin{align}
\begin{split}
\label{eq:residual-express}
\sin^2(v_1, \oja)  &=
1-\frac{u_0^T B_n^T v_1v_1^TB_n u_0}{u_0^TB_n^TB_n u_0} 
= \frac{u_0^T B_n^T (I-v_1v_1^T)B_n u_0}{u_0^TB_n^TB_n u_0} 
 \\ &= \frac{ (V_\perp V_\perp^T B_n u_0)^T (V_\perp V_\perp^T B_n u_0) }{\|B_n u_0\|^2}
\end{split}
\end{align} 
where in the last line we used the fact that $V_\perp V_\perp^T$ is idempotent.  Our proof strategy for the central limit theorem involves further approximating Eq \ref{eq:residual-express} with an inner product of the H\'{a}jek projection (first-order) term in Eq \ref{eq:hayek-proj-decomp}.     
\item Our second step is to show that $\|B_n u_0\|$ concentrates around its expectation $(1+\eta_n\lambda_1/n)^n|v_1^T u_0|$.
    \item Next we establish that  $\frac{\|\vorth\vorth B_n\vorth\vorth^T u_0\|_2}{\|B_n u_0\|}$ is $O_P\left( \sqrt{d} \cdot \exp\{- \eta_n(\lambda_1 - \lambda_2)\}  + \sqrt{\frac{\eta_n^3 M_d^2\log d}{n^2}} \right)$. This is achieved by using a similar recursive argument as in~\cite{jain2016streaming}, but with the crucial observation that the residual or common difference term is of a lower order because it can be replaced by a matrix product minus its expectation.
    \item Now we go back to the expansion in Eq~\ref{eq:hayek-proj-decomp}. 
    \bas{
    (v_1^Tu_0)\vorth\vorth^T B_n v_1 = (v_1^Tu_0)\sum_k \vorth\vorth^T T_k v_1
    }
    Since $T_0=(I+\eta_n/n\Sigma)^n$,  $\vorth\vorth^T T_0 v_1 $ is the zero vector. Now we examine the $(v_1^Tu_0)\vorth\vorth^T(B_n-T_1)v_1$ term. Here we use the structure of the higher order terms $T_k$. In particular, we use the fact that it is a matrix product interlaced with $k$ $X_iX_i^T-\Sigma$ matrices. For example, for $k=2$ we have  
    \bas{
    T_2=\frac{\eta_n^2}{n^2}\sum_{i<j}\bb{I+\frac{\eta_n}{n}\Sigma}^{i-1}Y_i\bb{I+\frac{\eta_n}{n}\Sigma}^{j-i-1}Y_j\bb{I+\frac{\eta_n}{n}\Sigma}^{n-j}
    }
     We show that the norm of $(v_1^Tu_0)\vorth\vorth^T(B_n-T_1)v_1$, normalized by the denominator, is $O(\eta_n^2 M_d^2/n^2)$. The fact that the summands of $T_k$ are uncorrelated and $T_k$ and $T_\ell$ are uncorrelated for $k\neq\ell$ makes this possible.
    \item Finally, we are left with $\vorth\vorth^T T_1 v_1(v_1^T u_0)$. Note that this is of the following form:
    \bas{
    \frac{\eta_n}{n}\frac{(v_1^T u_0)\vorth\vorth^T T_1 v_1}{|v_1^T u_0|(1+\lambda_1\eta_n/n)^n} = \frac{\eta_n\sgn(v_1^Tu_0)}{n}\sum_{i=1}^n \vorth\Lambda_\perp^{i-1}\vorth^T (X_iX_i^T-\Sigma)v_1
    }
    It is not hard to see that this is a sum of independent random vectors with covariance matrix $\eta_n/n \bigvar$ (see Eq~\ref{eq:bigvar}). 
    \item We  adapt a result of distributional convergence of squared norm of sums of IID random vectors in~\cite{10.1093/biomet/asz020} to squared norm of sums of independent random vectors. Under the assumptions~\ref{eq:clt-norm-bound} and~\ref{eq:clt-lb}, the conditions of distributional convergence are satisfied. 
    \item Finally, all the error terms are combined along with an anti-concentration argument for $\chi^2$ to establish the final result. The full proof and accompanying lemmas are in Section~\ref{sec:B} of the Supplement.
\end{enumerate}
\end{proof}

\subsection{Bootstrap consistency}
Using the weighted $\chi^2$ approximation for inference requires estimating the eigenvalues of $\Sigma$ and other population quantities; however, accurate estimates may not be available in a streaming setting. Instead,  we propose a streaming bootstrap procedure that mimics the properties of the original Oja algorithm. While a similar structure leads to error terms that are similar to the CLT, the analysis of the bootstrap presents its own technical challenges.  In what follows let $P^*$ denote the bootstrap measure, which is conditioned on the data, and let $\mathbb{E}^*[\cdot]$ denote the corresponding expectation operator. 

A common strategy for establishing consistency of the Gaussian multiplier bootstrap is to invoke a Gaussian comparison lemma.  Since the multipliers are themselves Gaussian and the data is treated as fixed, the idea is that one can use specialized results for comparing the distributions of two Gaussians (bootstrapped $Z_n^*$ and approximating $Z_n$ from the CLT) that only depend on how close the covariance matrices $\E^*[ Z_n^*Z_n^{*T}]$ and $\E [Z_nZ_n^T]$ are in an appropriate metric.  Using a Gaussian comparison lemma for quadratic forms (see Supplement Section \ref{sec:gausscompare}), we have the following result for the bootstrapped $\sin^2$ error:

\begin{lemma}\label{lem:covdiff}[Bounding the difference between the bootstrap covariance and true covariance]
Let:
\ba{\label{eq:zstar}
Z_n^*=\sgn(v_1^Tu_0)\sqrt{\frac{\eta_n}{n}} \sum_i W_i \vorth \Lambda_{\perp}^{i-1}\vorth^T(X_iX_i^T-X_{i-1}X_{i-1}^T)v_1.
}
Recall the definition of $\bigvar$ from Eq~\ref{eq:bigvar}.
We have,
\bas{
|\trace(\E^*[Z_n^*Z_n^{*T}]-\bigvar)|,\|\E^*[Z_n^*Z_n^{*T}]-\bigvar\|_F&=O_P\bb{\sqrt{\frac{\E\|X_1X_1^T-\Sigma\|^4}{n(\lambda_1-\lambda_2)}}} 
}
\end{lemma}
 With this lemma in hand, we are ready to state our bootstrap result.  
  
\begin{theorem}[Bootstrap Consistency]
\label{thm-oja-boot}
Suppose that the conditions of Theorem \ref{thm-oja-clt} are satisfied.  Furthermore, let $\alpha_n$ be a sequence such that $P(\mathcal{A}_n^c) \rightarrow 0$, where $\mathcal{A}_n$ is defined as $\mathcal{A}_n =  \left\{\max_{i \leq i \leq n } \|X_i\|^2 \leq \alpha_n \right\}$. Further suppose that $\frac{M_d \log^2{d} \ \eta_n^2}{n} \rightarrow 0$, $\frac{( \alpha_n^3 \vee M_d \log d) \ \alpha_n \eta_n^3}{n} \rightarrow 0$, $\frac{\alpha_n M_d \eta_n^2}{n(\lambda_1-\lambda_2)} \rightarrow 0 $, and $\frac{\mathbb{E}[\norm{X_1X_1^T- \Sigma}^4]}{n(\lambda_1-\lambda_2)} \rightarrow 0 $.
Then,
\begin{align*}
 \sup_{t \in \mathbb{R}} \left|P^*(n/\eta_n \sin^2( v_1^*, \hat{v}_1) \leq t) - P(n/\eta_n \sin^2( \hat{v}_1, v_1) \leq t)  \right| \xrightarrow{P} 0   
\end{align*}
\end{theorem}

\begin{proof}[Proof sketch of Theorem~\ref{thm-oja-boot}]
The proof follows a similar route to Theorem~\ref{thm-oja-boot}. We provide a detailed analysis in Supplementary Section. We use a bootstrap version of the  Hoeffding decomposition conditioned on the data, stated in Supplement Section. In step one we have $B_n^*$ replace $B_n$, where $B_n^*$ is given by:
\bas{
B_n^*=\prod_{i=1}^n \bb{I+\eta_n/n (X_iX_i^T+W_i(X_iX_i^T-X_{i-1}X_{i-1}^T)}
}
We work out Step 1 using concentration of matrix products~\cite{huang2020matrix}. For steps 2-3, we see that $T_k^*$ has the same structure as $T_k$ with the difference that $(I+\eta_n\Sigma/n)^i$ is replaced by its sample counterpart which is a product of $i$ independent matrices of the form $I+\eta_n/n X_jX_j^T$.  Concentration of these terms in operator norm are established with results from~\cite{huang2020matrix}. Finally for step 4, we see that the main term that approximates the bootstrap residual $\vp\vp^T B_n^* u_0$ is given by $\sqrt{\eta_n/n}Z_n^*$, where $Z_n^*$ is given in Eq~\ref{eq:zstar}.
Conditioned on the data, this is already Normally distributed since the multiplier random variables $W_i$ are themselves Gaussian. We then invoke the Gaussian comparison result Lemma \ref{lem:covdiff} to obtain  convergence to the weighted  $\chi^2$ approximation.
\end{proof}

We now make a couple of points regarding our analysis. It should be noted that the terms in the product are weakly dependent, which is different from the CLT and would seem to complicate concentration arguments used to establish bootstrap consistency.  However, the dependence is not strong and second-moment methods may be used.  We also operate on a good set in which the norms of the the updates are not too large, which is far less restrictive than assuming an almost sure bound on the norm.   

In theorem above, we have stated the good set $\mathcal{A}_n$ in an abstract manner, but one may wonder how stringent the condition is in various problem settings.  Below, we describe a general setup with sub-Gaussian entries of $X_i$  in which $\alpha_n$ grows as $\log n$; under milder forms of various decay, all we need is for $\alpha_n$ to grow slowly with $n$. Here $\norm{\cdot }_{\psi_1}$ is the sub-Exponential Orlicz norm and $\norm{\cdot }_{\psi_2}$ is the sub-Gaussian Orlicz norm  (see, for example~\cite{verhsynin-high-dimprob}).
\begin{proposition}[The effect of variance decay on the norm]\label{prop:vardecay}
For each $1 \leq j \leq p$, suppose that $X_{1j}$ satisfies $\norm{X_{1j}}_{\psi_2} \leq \nu_j$ $\sum_{j=1}^p \nu_j \leq C_1 < \infty$.  Then, for some universal constant $ C_2>0$, $\norm{\sum_{j=1}^p(X_{1j}^2 - \mathbb{E}[X_{1j}^2]) }_{\psi_1} < C_2$, and for some $c_1, c_2> 0 $,
\begin{align*}
P\left( \max_{1 \leq i \leq n} \norm{X_{i}}^2 > c_1 \log n \right) \leq \frac{c_2}{n}     
\end{align*}
\end{proposition}
We now present experimental validation of our bootstrap procedure below.  
\section{Experimental validation of the online multiplier bootstrap}
\label{sec:exp}
We draw $Z_{ij}\stackrel{IID}{\sim} \mathrm{Uniform}(-\sqrt{3},\sqrt{3})$, for $i=1,\dots, n$ and $j=1,\dots d$.  Consider a PSD matrix $K_{ij}=\exp(-|i-j|c)$ with $c=0.01$. We create a covariance matrix such that $\Sigma_{ij}=K(i,j)\sigma_i\sigma_j$. We consider $\sigma_i=5 i^{-\beta}$ for $\beta = 0.2$ and $\beta = 1$. Now we transform the data to introduce dependence by letting $X_i =\Sigma^{1/2} Z_i$.  By construction, we have that $\E[X_{i} X_{i}^T] = \Sigma$ for all $1 \leq i \leq n$.  Our goal is to simply demonstrate that the bootstrap distribution of $\sin^2$ errors closely match that of the sampling distribution. To this effect, we fix $u_0$ and draw $500$ datasets and run streaming PCA on each and then construct an empirical CDF ($F$) from the $\sin^2$ error with the true $v_1$. This is the point of comparison for the bootstrap distribution ($F^*$), for which we fix a dataset $X$.  We then invoke algorithm~\ref{alg:boot} to obtain 500 bootstrap replicates  $\hat{v}^{*}_1$ as well as the Oja vector for the dataset $\oja$. The bootstrap distribution is the empirical CDF of $1-(\oja^T\hat{v}^{*}_1)^2$. We use $\eta_n=\log n$.
\begin{figure}[h]
    \centering
    \begin{tabular}{cc}
         \includegraphics[width=0.4\textwidth]{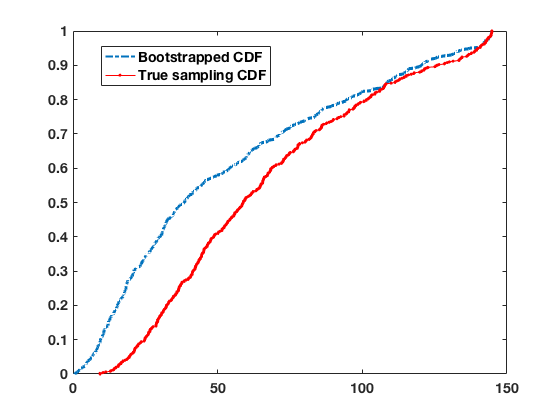}& \includegraphics[width=0.4\textwidth]{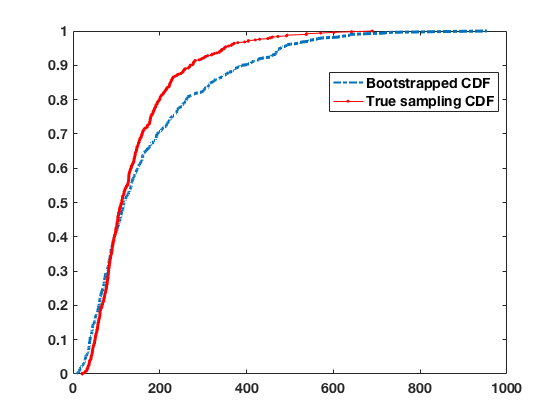} \\
         (A)&(B)\\
         \includegraphics[width=0.4\textwidth]{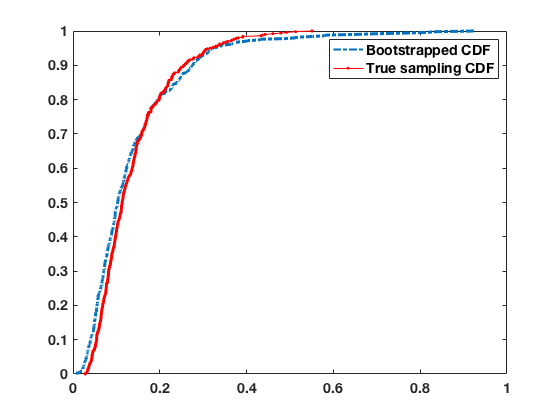}& \includegraphics[width=0.4\textwidth]{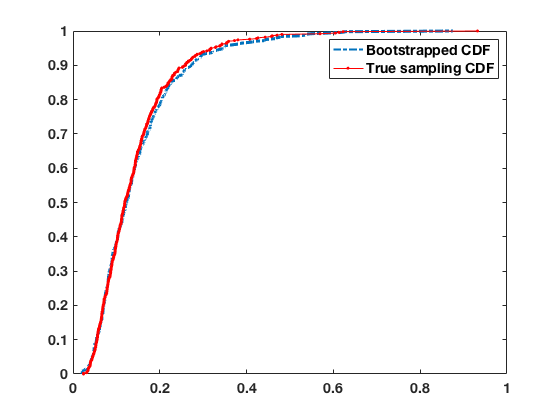} \\
         (C)&(D)
    \end{tabular}
    \caption{\label{fig:chisq}  Bootstrapped and sampling CDF for $n=1000,d=500$ in (A) and (C) and for $n=10,000,d=500$ in (B) and (D). (A) and (B) use $\beta=0.2$ whereas (C) and (D) use $\beta=1$.}
\end{figure}
In Figure~\ref{fig:chisq},
we see that for $\beta=0.2$ (see (A) and (B)), where the variance decay is slow and therefore the error bounds of the residual terms are expected to be large, the quality of approximation is poorer compared to (C) and (D), where $\beta=1$. However, even for $\beta=0.2$, increasing $n$ improves performance. Also note that, for (A) and (B) the variance decay does not satisfy our theorem’s conditions and thus, the normalized error does not behave like a $O_P(1)$ random variable. However, for (C) and (D) the variance decay satisfies the conditions and in this case the normalized error is $O_P(1)$, which happens to be in the [0,1] range for this example.

\section{Discussion}
Modern tools in non-asymptotic random matrix theory have given rise to recent breakthroughs in establishing pointwise convergence rates for  stochastic iterative methods in optimizing certain nonconvex objectives, including the classic Oja's algorithm for online principal component analysis.  By synthesizing modern random matrix theory tools with classic results from the  U-statistics literature and recently developed high-dimensional central limit theorems, we extend the error analysis of Oja's algorithm from pointwise convergence rates to  \emph{distributional} convergence and moreover establish an efficient online bootstrap method for Oja's algorithm to quantify the error on the fly.   Our results are a first step toward incorporating uncertainty estimation into the general framework of stochastic optimization algorithms, but we acknowledge the present limitations of our analysis: new tools will be needed to extend the current analysis to estimating higher-dimensional principal subspaces, and additional tools will be needed to account for non-independent matrix products which appear beyond the setting of online PCA. 

\section*{Acknowledgment}
P.S. and R.L. are supported in part by NSF 2019844 and NSF HDR-1934932.
R.W. is supported in part by AFOSR MURI FA9550-19-1-0005, NSF DMS 1952735, NSF HDR-1934932, and NSF 2019844.


 \bibliographystyle{abbrv}
\bibliography{references}

 \makeatletter\@input{zzz.tex}\makeatother

\end{document}


\maketitle

\section*{Supplementary Material} 
\setcounter{section}{0}
\setcounter{equation}{0}
\renewcommand{\thesection}{\Alph{section}}
\renewcommand{\theequation}{S.\arabic{equation}}
\renewcommand{\thefigure}{S.\arabic{figure}}
\renewcommand{\thetable}{S.\arabic{table}}

\setcounter{theorem}{0}
\setcounter{proposition}{0}
\setcounter{corollary}{0}
\setcounter{lemma}{0}

\renewcommand{\thetheorem}{\Alph{section}.\arabic{theorem}}
\renewcommand{\thelemma}{\Alph{section}.\arabic{lemma}}

\renewcommand{\theproposition}{\Alph{section}.\arabic{proposition}}

\renewcommand{\thecorollary}{\Alph{section}.\arabic{corollary}}

In this document we provide the detailed proofs of results presented in the main manuscript. 
 In Section~\ref{sec:HD}, we provide a proof for the Hoeffding expansion of the matrix product in Eq~5 of the main document.  We also provide the Hoeffding decomposition for the bootstrap in Proposition~\ref{prop:boothoeff}.  In Section~\ref{sec:B} we provide all results needed for a complete proof of Theorem~\thmojaclt.  In Sections~\ref{sec:ojaclt}, ~\ref{sec:CLTadapt}, and~\ref{sec:CLT-support} we provide the proof of Theorem~\thmojaclt, the adaptation of high dimensional CLT of~\cite{10.1093/biomet/asz020} to our setting and all supporting lemmas, respectively. 
    
In Section~\ref{sec:C} we provide all details of the proof of the Bootstrap consistency, i.e. Theorem~\ref{thm-oja-boot}. To be specific, Section~\ref{sec:bootproof} has the proof of Theorem~2; Section~\ref{sec:lemcovdiff} has the proof of Lemma~\ref{lem:covdiff}, Section~\ref{sec:gausscompare} has the statement and proof of the Gaussian comparison lemma, and Section~\ref{sec:bootsupplem} has all the supporting lemmas. Finally, in Section~\ref{sec:D}, we provide a proof of Proposition~\ref{prop:vardecay}.

\section{On the Hoeffding decomposition}\label{sec:HD}
We discuss Hoeffding decompositions for a function $f$ of $n$ independent random variables $X_1, \ldots X_n$, where the random variables take values in an arbitrary space and the function takes values\footnote{The math generalizes to Hilbert spaces due to the Hilbert projection theorem but we specialize to these cases for concreteness.} in $\mathbb{R}^{d\times d}$ or $\mathbb{R}^d$. The following exposition largely follows \cite{vandervaart-asymptotic}.

With Hoeffding decompositions, we project  $T(X_1,\ldots, X_n)$ onto spaces of increasing complexity that are orthogonal to each other.  In our setup, orthogonality means $\langle f, g \rangle_{L^2} = 0$ where $\langle f,g \rangle_{L^2} = \int \langle f, g \rangle dP$.  Here,  $\langle f,g \rangle = \mathrm{Trace}(f^Tg)$ in the matrix case and $\langle f,g \rangle = f^Tg$ in the vector case. The first-order projection, also known as a H\'{a}jek projection, involves projecting our function onto a space of functions of the form
\begin{align*}
g^{(i)}(X_i)    
\end{align*}
where $g^{(i)}$ satisfies $E[g^{(i)}] = 0$.  We will let $H^{(i)}(X_i)$ denote the corresponding projection. Since the functions $g^{(i)}, g^{(j)}$ are mutually orthogonal for $i \neq j$, the sum of the projections is equivalent to the projection onto the space spanned by functions of the form:
\begin{align*}
\sum_{i=1}^n g^{(i)}(X_i) 
\end{align*}
The higher-order spaces have the form:
\begin{align*}
 g^{(S)}( X_i : i \in S)   
\end{align*}
where $S \subseteq \{1, \ldots, n \}$ and the functions satisfy $\mathbb{E}[g^{(S)} \ | \ X_i : i \in R] = 0$ for any $R \subset S$, including $R = \emptyset$, which implies $\mathbb{E}[g^{(S)}] = 0$.  If  $ R \not \subset S$ and $S \not \subset R$, $\langle g^{(S)}, g^{(R)} \rangle_{L^2}=0$ since, by conditional independence given $\{X_i : i \in R \cap S \}$:
\begin{align} 
\label{eq-orthogonality-spaces}
\mathbb{E}[\mathbb{E}[\langle g^{(S)}, g^{(R)} \rangle \ | \  X_i : i \in  R \cap S  ] \ ]  = 
 \mathbb{E}\left[ \bigl\langle \mathbb{E}[ g^{(S)} \ | \ X_i : i \in R \cap S ]  , \  \mathbb{E} [ g^{(R)} \ | \  X_i : i \in  R \cap S ] \ \bigr\rangle  \right] = 0
\end{align}
Combining these projections leads to the following representation, known as the Hoeffding decomposition:
\begin{align*}
T(X_1, \ldots, X_n) &= \sum_{k=0}^n \sum_{|S| =k } H^{(S)}( X_i : i \in S)   
\end{align*}

While the following proposition is stated for real-valued functions in \cite{vandervaart-asymptotic}[Lemma 11.11], it turns out that the proof there generalizes to our setting without difficulty due to machinery for projections in  Hilbert spaces. 

\begin{proposition}[Hoeffding projections]
Let $X_1, \ldots, X_n$ be arbitrary random variables and let suppose $\langle T,T \rangle_{L_2} < \infty$.  Then the projection on the the space of functions of the form $g^{(S)}( X_i : i \in S)$ with $\mathbb{E}[g^{(S)} \ | \ X_i : i \in R] = 0$ for any $R \subset S$ has the form:
\begin{align*}
H^{(S)}(T) = \sum_{R \subseteq S} (-1)^{|S|-|R|} \ \mathbb{E}\left[T \ | \ X_i: i \in R \right]
\end{align*}
\end{proposition}

For completeness, we provide a proof of the proposition below. 
\begin{proof}
We begin by verifying that the space of all random matrices (vectors) satisfying  $\|A\|_{L^2} < \infty$ forms a Hilbert Space.  First, it is clear that $\langle \cdot , \cdot \rangle_{L^2}$ is indeed an inner product.  Linearity follows from linearity of the inner product $\langle \cdot , \cdot \rangle$ and linearity of expectations and conjugate symmetry follows from this property holding pointwise in $\Omega$ for $\langle \cdot , \cdot \rangle$.  Positive definiteness again follows from the fact that this property holds pointwise in $\Omega$; then a standard contradiction argument yields that if $\langle x,x\rangle_{L^2} =0$, but $x$ is not equal to 0 almost surely, there exists some $M$ such that for some $\delta > 0$, $P(\|x\| > \frac{1}{M}) \geq \delta $ and hence $\int \langle x,x \rangle dP \geq \delta/M > 0$, a contradiction.

One can again adapt standard arguments for completeness of $L_2$ spaces to our setting; namely, show that Cauchy sequences converging in $L_2$ implies convergence almost everywhere, and then invoke completeness of the Hilbert space over matrices/vectors along with integral convergence theorems; see for example, the proof of Theorem 1.2, page 159 in \cite{Stein-Shakarchi-Analysis}.

Now to verify that this function is indeed the projection, we invoke the Hilbert Projection Theorem; see for example, Lemma 4.1 of \cite{Stein-Shakarchi-Analysis}.  To use this theorem, we need to check that the space spanned by functions of the form $g^{(S)}$ satisfying the condition $\mathbb{E}[g^{(S)} \ | \ X_i: i \in R] = 0$ for any $R \subset S$ is a closed subspace.  Linearity of the space follows from the fact that the sum of such functions satisfies the constraint; therefore it is a subspace.  To check closure, let $\| f \|^2 = \langle f,f \rangle$ and consider some (convergent) sequence in this subspace $(g_\alpha^{(S)})_{\alpha \geq 1}$ where $g_\alpha^{(S)} \rightarrow g^{(S)}$ and observe that, for any $R \subset S$:
\begin{align*}
\mathbb{E}[\|g_\alpha^{(S)} - g^{(S)}      \|^2] &= \mathbb{E}[ 
\ \mathbb{E}[ \| g_\alpha^{(S)} - g^{(S)}      \|^2 \ | \ X_i : i \in R] \ ]
\\ & \geq \mathbb{E}\left[\|\mathbb{E}[g_\alpha^{(S)} - g^{(S)} \ | \ X_i : i \in R]      \|^2  \right]
\\ & \geq \mathbb{E}\left[ \|\mathbb{E}[g^{(S)} \ | \ X_i : i \in R]\|^2\right]
\end{align*}
where above we used the fact that $\mathbb{E}[g_\alpha^{(S)} \ | \ X_i : i \in R] = 0$ for all $\alpha$ by assumption. Since the LHS converges to 0, it follows that $\mathbb{E}[g^{(S)} \ |  X_i : i \in R]$ must be equal to $0$ almost surely. Since the limit satisfies $\mathbb{E}[g^{(S)} \ |  X_i : i \in R] =0$ for all $R \subset S$, it belongs in the space, proving closure.

Now, we show that the stated expression is indeed the Hoeffding projection.  First, to show that belongs in this space, we have, following analogous reasoning to \cite{vandervaart-asymptotic}, for any $C \subset A$,
\begin{align*}
\mathbb{E}[ H^{(A)}(T) \ | \ X_i : i \in C] &= \sum_{B \subseteq A} (-1)^{|A| - |B|} \mathbb{E}[ T \ | \ X_i: i \in B \cap C ]
\\ &=  \sum_{D \subseteq C} \sum_{j=0}^{|A|-|C|} (-1)^{|A| - (|D|+j) } {|A|-|C| \choose j } \mathbb{E}[ T \ | \ X_i: i \in D ]
\\ &= \sum_{D \subseteq C} (-1)^{|C| - |D|} \  \mathbb{E}[ T \ | \ X_i: i \in D ] \ (1-1)^{|A|- |C|} = 0
\end{align*}
where the last line follows from the Binomial Theorem. Now as a consequence of the Hilbert Projection Theorem, it suffices to show that $ H^{(A)}(T)$ satisfies the property:
\begin{align*}
\langle  T- H^{(A)}(T), g^{(A)} \rangle_{L^2} = 0   
\end{align*}
for any $g^{(A)}$ in the space.  In the matrix case, we have
\begin{align*}
\langle T - H^{(A)}(T) , g^{(A)} \rangle_{L^2} =&  \sum_{j=1}^d \sum_{k=1}^d \mathbb{E}[(T_{jk} - \mathbb{E}[T_{jk} \ | \ X_i : i \in A]) \cdot g_{jk}^{(A)}]
 \\ + & \sum_{j=1}^d \sum_{k=1}^d \sum_{B \subset A} \mathbb{E}\left[ (-1)^{|A|-|B|}  \mathbb{E}[T_{jk} \ | \ X_i : i \in B ] \cdot \mathbb{E}[g_{jk}^{(A)} \ | \ X_i : i \in B]  \right]
\end{align*}
The first term above is $0$ since conditional expectations may be viewed as an orthogonal projection in the Hilbert Space with inner product $\int f g \ dP$ into the closed subspace of $\sigma( X_i : i \in A)$-measurable functions.  The second term is zero since $\mathbb{E}[g_{jk}^{(A)} \ | \ X_i : i \in B] = 0$ for any $B \subset A $.  The vector case is analogous.

Since this property holds, it must be the unique (up to measure 0 sets) minimizer and projection. 
\end{proof}
Now an immediate corollary for our setting follows.
\begin{proposition}[Orthogonality of Hoeffding projections]\label{prop:orth-hoeffding-proj}
Let:
\begin{align*}
B_n = \sum_{k=0}^n \sum_{|S|=k} H^{(S)} 
\end{align*}
where $A^{(S)}$ is the Hoeffding projection corresponding to the set $S \subseteq \{1, \ldots, n \}$.  Then,
 \begin{align*}
 \mathbb{E}\left[\norm{B_n}_F^2 \right] &= \sum_{k=0}^n \sum_{|S| = k } \mathbb{E}\left[\|A^{(S)}\|_F^2\right]\\  \mathbb{E}\left[ \|B_n x\|^2 \right] &= \sum_{k=0}^n \sum_{|S| = k } \mathbb{E}\left[\| A^{(S)}x\|^2\right]  
 \end{align*}
 where the last inequality holds for all $x \in \mathbb{R}^d$.
 \end{proposition}
 \begin{proof}
  Letting $ g^{(S)} =  H^{(S)}$ and $g^{(R)} = H^{(R)}$ in Eq \ref{eq-orthogonality-spaces}, we have that $\langle H^{(S)}, H^{(R)} \rangle_{L^2} = 0$ for all $R \neq S$ and the result follows. 
 \end{proof}
 It remains to be shown that Hoeffding decomposition has the form stated in Eq \ref{eq:hayek-proj-decomp}.  Deriving all projections in the Hoeffding decomposition for a general function is typically non-trivial, but the product structure facilitates our proof below.
Before establishing the Hoeffding decomposition, following for example, \cite{bentkus-symmetric-edgeworth} observe that the following inverse relation holds:
\begin{proposition}[Conditional expectation and Hoeffding projections]
\label{prop-cond-exp-hoeffding}
 \begin{align*}
\mathbb{E}\left[ T \ | \ X_i : i \in S  \right] = \sum_{R \subseteq S} H^{(R)}(T)     
 \end{align*}
 \begin{proof}
 Observe that:
 \begin{align*}
\mathbb{E}[T \ | X_i : i \in S ] &=\sum_{k=0}^n \sum_{|R| =k }  \mathbb{E}[H^{(R)}(T) \ | \  X_i : i \in S]       
 \end{align*}
 Since the conditional expectation is zero for $R \not \subseteq S$ and for $R \subseteq S$, the Hoeffding projection is fixed, the result follows.
 \end{proof}
\end{proposition}
Now we are ready to establish the form of the Hoeffding projection for any $S \subseteq \{1, \ldots,n \}$.  We in fact prove a slightly stronger statement, which makes the induction argument more natural.  In what follows let $S[i]$ denote the $i$th element in $S$. We will also use $H^{(S)}$ instead of $H^{(S)}(T)$ when it is clear from the context.
\begin{theorem}[\label{thm:oja-hoeffding}Hoeffding projections for Oja's algorithm]
\label{thm-bootstrap-hoeff-proj}
Define:
\begin{align*}
T_{-j} = \prod_{i=j+1}^n \left( I+ \frac{\eta_n }{n} X_i X_i^T \right) , \ \ \ T = T_{-0} =  \prod_{i=1}^n \left( I+ \frac{\eta_n }{n} X_i X_i^T \right),
\end{align*}

Then for any $S \subseteq \{1, \ldots, n\}$ and for all $0 \leq j < S[1]$, we have the Hoeffding projection of $T_{-j}$ onto $\{ X_i : i \in S\}$ may be expressed as:
\begin{align}\label{eq:hminusj}
H_{-j}^{(S)} = \prod_{i=j+1}^n A_i^{(S)}, \ \ H^{(S)} = H_{-0}^{(S)} 
\end{align} 
where:
\begin{align*}
A_i^{(S)}  = \begin{cases}
\frac{\eta_n}{n} (X_i X_i^T - \Sigma) & i \in S \\ 
I+ \frac{\eta_n}{n} \Sigma  & i \not\in S
\end{cases}
\end{align*}
\end{theorem}
\begin{proof}
We will conduct (strong) induction on $k = |R|$, where $R \subseteq S$.  We will start with the base case $k=1$;  $k=0$ is simply the expectation.  For the base case $|R| =1$, a direct calculation is possible, since:
\begin{align*}
H_{-j}^{(R)} &= \mathbb{E}[ T_{-j} \ | \ X_i : i \in R] - \mathbb{E}[T_{-j}],
\end{align*}
which has the stated form.  Now, we will suppose that the inductive hypothesis holds.  In what follows, let $S[1] = k$ and define the conditional expectation for any set $S$ as:
\begin{align*}
\mathbb{E}\left[T_{-j} \ | \ X_i : i \in S \right] = \prod_{i=j+1}^n E_i^{(S)},
\end{align*}
where:
\begin{align*}
E_i^{(S)} = \begin{cases}
I + \frac{\eta_n}{n} X_i X_i^T  & i \in S \\ 
I +  \frac{\eta_n}{n} \Sigma &  i \not \in S
\end{cases}
\end{align*}
We will now add and subtract  a product where an entry corresponding to $S[1]$ in $\mathbb{E}[T_{-j} \ | \ X_i : i \in S ]$ is  replaced by $(I+ \frac{\eta_n}{n} \Sigma)$.  Doing, so we have  
\begin{align*}
\mathbb{E}[T_{-j} \ | \ X_i : i \in S] = & \  \mathbb{E}\left[T_{-j} \ | \ X_i : i \in  S \right]  - (I+ \frac{\eta_n}{n} \Sigma)^{k-j}  \times \prod_{i=k+1}^n E_i^{(S)}  \\ 
 & + (I+ \frac{\eta_n}{n} \Sigma)^{k-j} \times \prod_{i=k+1}^n E_i^{(S)}
\end{align*}
We recognize the second summand as $\mathbb{E}[T_{-j} \ | \ X_i :  i \in S_{-k} ]$, where $S_{-k} = \{ i \in S, i \neq k\}$.   Now for the first summand, taking the difference we have the term
\begin{align*}
& (I+ \frac{\eta_n}{n} \Sigma)^{k-j-1} \times \frac{\eta_n}{n} (X_k X_k^T - \Sigma)  \times \prod_{i=k+1}^n E_i^{(S)}
\\ = & \ (I+ \frac{\eta_n}{n} \Sigma)^{k-j-1} \times \frac{\eta_n}{n} (X_k X_k^T - \Sigma) \times \mathbb{E}\left[T_{-k} \ | \ X_i : i \in S_{-k} \right] 
\end{align*}
By Proposition \ref{prop-cond-exp-hoeffding}, we may represent a conditional expectation as:
\begin{align}\label{eq:bentkus_repr}
\mathbb{E}\left[ T_{-k} \ |  \ X_i : i \in S_{-k} \right] &= \sum_{ R \subseteq S_{-k} } H_{-k}^{(R)}
\end{align} 

Furthermore, by the inductive hypothesis, each $ H_{-k}^{(R)}$ takes the form in Eq~\ref{eq:hminusj}. Now, combining the two parts, we have 
\begin{align*}
 \ \mathbb{E}[T_{-j} \ | \ X_i : i \in S ] =& \ \sum_{ R \subseteq S_{-k} } (I+ \frac{\eta_n}{n} \Sigma)^{k-j-1} \times \frac{\eta_n}{n} (X_k X_k^T - \Sigma) \times  H_{-k}^{(R)}
\\  + &  \ \sum_{ R \subseteq S_{-k} } (I+ \frac{\eta_n}{n} \Sigma)^{k-j} \times  H_{-k}^{(R)}
\\ = & \  \prod_{i=j+1}^n A_i^{(S)} +  \sum_{ R \subset S } H_{-j}^{(R)}
\end{align*}
For the last step, notice that with the exception of $ R = S_{-k}$ in the first sum, each product in the sum corresponds to a Hoeffding projection of  some set of size less than $k$ by the inductive hypothesis. The first term must be the Hoeffding projection onto $S$ (with $S[1]=k> j$) by the same argument as Eq~\ref{eq:bentkus_repr}, i.e. 
\bas{
H_{-j}^{(S)}=\prod_{i=j+1}^n A_i^{(S)},
}
proving the desired result. 
\end{proof}
Now, since the Hoeffding decomposition is a sum of Hoeffding projections by definition, we have the following corollary.
\begin{corollary}[Hoeffding decomposition for Oja's algorithm]
\label{cor-hoeff-decomp-oja}
\begin{align*}
B_n = \sum_{k=0}^n \sum_{|S| = k} H^{(S)}    
\end{align*}
where $A^{(S)}$ is given by $H^{(S)}$ in Eq \ref{eq:hminusj}. 
\end{corollary}

It turns out that the bootstrap Hoeffding decomposition can be proved using the same strategy in Theorem \ref{thm-bootstrap-hoeff-proj}, where $X_1, \ldots, X_n$ is treated as fixed in the bootstrap measure.  We state the result below.

\begin{proposition}[Hoeffding decomposition for the bootstrap]
\label{prop:boothoeff}
\begin{align*}
B_n^* = \sum_{k=0}^n \sum_{|S| = k} \alpha^{(S)}    
\end{align*}
where $\alpha^{(S)} = \prod_{i=1}^n \alpha_i^{(S)}$ and  $\alpha_{i}^{(S)}$ is given by: 
\begin{align*}
\alpha_{i}^{(S)} &= \begin{cases}
 \frac{\eta_n}{n} W_i \cdot (X_iX_i^T - X_{i-1}X_{i-1}^T ) & \text{if} \ i \in S  \\ 
 I + \frac{\eta_n}{n} X_iX_i^T & \text{otherwise}
\end{cases}
\end{align*}

\end{proposition}

\section{Central limit theorem for Oja's algorithm}
\label{sec:B}

\subsection{Proof of Theorem~\thmojaclt}\label{sec:ojaclt}

\begin{proof}[Proof of Theorem \thmojaclt]
Our strategy will be to approximate $\sin^2$ distance for estimated eigenvector with a quadratic form, and invoke a high-dimensional central limit theorem result.  The remainder terms will be bounded using an anti-concentration result for weighted $\chi^2$ random variables due to \cite{10.1093/biomet/asz020}.

Observe that $\sin^2(\hat{v}_1, v_1)$ has the representation:
\begin{align*}
1- \left(v_1^T\frac{B_n u_0}{\norm{B_n u_0}} \right)^2 = \frac{u_0^T B_n^T(I - v_1v_1^T)B_n u_0}{\norm{B_nu_0}^2}  
\end{align*}

Let $V_{\perp}V_{\perp}^T = I - v_1v_1^T$.  Clearly, $V_{\perp}V_{\perp}^T$ is idempotent and is a projection matrix, implying that it is also symmetric.  Therefore,
\begin{align}
\label{eq-inner-prod-form}
\frac{n}{\eta_n} \cdot \sin^2(u_n, v_1) =  \frac{(\sqrt{n/\eta_n} V_{\perp}V_{\perp}^T B_n u_0)^T(\sqrt{n/\eta_n} V_{\perp}V_{\perp}^T B_n u_0)}{\norm{B_n u_0}^2}
\end{align} 

Let $a_1 = (v_1^T u_0)$ denote the scalar projection of $u_0$ so that $u_0 =  a_1 v_1 + w$, where $w$ is in the orthogonal complement of $v_1$.

Our first reduction of (\ref{eq-inner-prod-form}) is to approximate the denominator with a more convenient quantity.  By Lemma \ref{lemma-norm-concentration}, we have that (\ref{eq-inner-prod-form}) may be written as
\begin{align*}
\frac{(\sqrt{n}/\eta_n \cdot V_{\perp}V_{\perp}^T B_n u_0)^T(\sqrt{n}/\eta_n \cdot V_{\perp}V_{\perp}^T B_n u_0)}{a_1^2(1 + \frac{\eta_n}{n} \lambda_1)^{2n}} \cdot R_1
\end{align*} 
where
\begin{align*}
R_1 = \frac{\norm{B_n u_0}^2}{a_1^2(1 + \frac{\eta_n}{n}  \lambda_1)^{2n}} = 1 - O_P\left(\sqrt{d} \ \mathrm{exp} \left(-\frac{\eta_n}{2}(\lambda_1-\lambda_2)\right) +   \sqrt{\frac{\eta_n^2 M_d \log d}{n}}\right)
\end{align*}
While the aforementioned Lemma is stated for $\frac{\norm{B_n u_0}}{|a_1|(1 + \frac{\eta_n}{n} \lambda_1)^{n}}$, the relationship holds for the squared quantity since with high probability for $n$ large enough, $|\frac{\norm{B_n u_0}}{|a_1|(1 + \frac{\eta_n}{n} \lambda_1)^{n}}| \leq 2$ and $|x^2-1^2| \leq 3 |x-1|$ for all $-2 \leq x \leq 2$.  

We will further approximate the quantity $\sqrt{n}/\eta_n \cdot V_{\perp}V_{\perp}^T B_n u_0$.  First we will bound the contribution of $V_{\perp}V_{\perp}^T B_n V_\perp V_\perp^T$.  By Lemma \ref{lem:CLTVperp} we have that:
\begin{align*}
R_2 : = \sqrt{\frac{n}{\eta_n}} \cdot \frac{V_\perp V_\perp^T B_n V_\perp V_\perp^T u_0}{|a_1|(1+\frac{\eta_n \lambda_1}{n})^{n}} = O_P\left( \sqrt{\frac{n d}{\eta_n}} \cdot \exp\{- \eta_n(\lambda_1 - \lambda_2)\}  + \sqrt{\frac{\eta_n^2 M_d^2\log d}{n}} \right) 
\end{align*}
Now it remains to bound the term $V_\perp V_\perp^T B_n v_1 (v_1^Tu_0)$. First, by Corollary \ref{cor-hoeff-decomp-oja}, $B_n$ can be decomposed as:
\begin{align*}
B_n = \sum_{k=0}^n T_k
\end{align*}
where for $S \subseteq \{1, \ldots, n\}$, $T_k$ is defined as:
\begin{align}
T_k = \sum_{ |S| = k } A^{(S)}
\end{align}
with $A^{(S)}$ taking the form in Eq \ref{eq:hminusj}.

Since $v_1$ is orthogonal to $V_{\perp}$:
\begin{align*}
\sqrt{\frac{n}{\eta_n}} \cdot \frac{V_{\perp}V_{\perp}^T T_0 \ v_1 a_1}{|a_1|(1+\eta_n/n \lambda_1)^n} = \sqrt{\frac{n}{\eta_n}} \cdot  \textrm{sign}(a_1) (I-v_1v_1^T)v_1 = 0.  
\end{align*}
%

%
Furthermore, by Lemma \ref{lemma-clt-hayek}, since $\frac{\eta_n^3M_d^2}{n} \rightarrow 0$
by assumption, 
\begin{align}
R_3 := \sqrt{\frac{n}{\eta_n}} \cdot \frac{V_{\perp}V_{\perp}^T (B_n-T_1)v_1 a_1}{|a_1|(1+\eta_n/n \lambda_1)^n} = O_P\left(\sqrt{ \frac{\eta_n^3M_d^2}{n}}\right)
\end{align}
%
Now our term of interest is given by:
\begin{align}
\label{eq-inner-product-expression}
\frac{(\sqrt{n/\eta_n} \cdot V_{\perp}V_{\perp}^T T_1 v_1)^T(\sqrt{n/\eta_n} \cdot  V_{\perp}V_{\perp}^T T_1 v_1)}{(1 + \frac{\eta_n}{n} \lambda_1)^{2n}}
\end{align} 
Now, observe that $ (I + \frac{\eta_n}{n} \Sigma)$ and $v_1 v_1^T$ share a common eigenspace and therefore commute.  Therefore, the terms in the product to the left of $T_1$ may be written as:
\begin{align}\label{eq:di}
\frac{V_\perp V_\perp^T (I + \frac{\eta_n}{n} \Sigma)^{i-1}}{(1+\frac{\eta_n}{n} \lambda_1)^{i-1}} = \sum_{j=2}^d \left(\frac{1 + \frac{\eta_n}{n} \lambda_j }{1 + \frac{\eta_n}{n} \lambda_1}\right)^{i-1} v_j v_j^T := D_{i-1}, \ \ \ \text{say}. 
\end{align}
Hence,
\begin{align*}
 \sqrt{\frac{n}{\eta_n}} \cdot \frac{V_{\perp}V_{\perp}^T T_1 v_1}{(1 + \frac{\eta_n}{n} \lambda_1)^{n}} &= \sqrt{\frac{\eta_n}{n}} \sum_{i=1}^n \left(1 + \frac{\eta_n}{n} \lambda_1\right)^{-1} D_{i-1}(X_iX_i^T - \Sigma)v_1
 \\ &=  \ S_n = \sqrt{n} \left(1 + \frac{\eta_n}{n} \lambda_1\right)^{-1} \frac{1}{n}\sum_{i=1}^n U_i, \ \ \ \text{say},      
\end{align*}
where
\begin{align}\label{eq:Ui}
    U_i=D_{i-1}(X_iX_i^T - \Sigma)v_1.
\end{align}
Observe that $S_n$ is a sum of independent but non-identically distributed random variables with mean $0$.  Therefore, if the conditions of Proposition \ref{prop:clt-adaptation} are satisfied, we may approximate $S_n^T S_n$ with $Z_n^TZ_n$, where $\mathbb{E}[Z_n] = 0$, $\mathrm{Var}(Z_n) = \mathrm{Var}(S_n)$.  Below define $\tilde{Z}_i$ to be a Gaussian vector with $\var(\tilde{Z}_i)=\var( (X_iX_i^T-\Sigma)v_1)$. Now define $Z_i=D_{i-1}\tilde{Z}_i$. 
We now verify these conditions.

First, we derive a lower bound on  $\norm{\bigvar}_F$  that will be used in all of the following bounds. Observe that $\norm{\bigvar}_F=\frac{\eta_n}{n}\norm{\sum_i \Lambda_\perp^{i-1}\bM\Lambda_\perp^{i-1}}_F$ and the $kl$th entry of $\sum_i \Lambda_\perp^{i-1}\bM\Lambda_\perp^{i-1}$ is lower bounded by:
\begin{align}
\begin{split}
\label{eq-lower-bound-V-bar}
&\frac{\eta_n}{n}\sum_{i\geq 1}\left(\frac{1+\eta_n\lambda_{k+1}/n}{1+\eta_n\lambda_1/n}\right)^{i-1}\left(\frac{1+\eta_n\lambda_{\ell+1}/n}{1+\eta_n\lambda_1/n}\right)^{i-1}\bM(k,\ell) 
\\ & \geq \frac{1 - \exp( -2\eta_n (\lambda_1 - \lambda_2)) \left(1- \frac{\eta_n^2 \lambda_1^2}{n} \right)^{-2}}{{2\lambda_1-(\lambda_{k+1}+ \lambda_{k+1}) + \frac{\eta_n}{n}(\lambda_1^2 - \lambda_k\lambda_l) }} \bM(k,\ell) 
\\ & \geq \frac{1 - \exp( -2\eta_n (\lambda_1 - \lambda_2)) \left(1- \frac{\eta_n^2 \lambda_1^2}{n} \right)^{-2}}{{2\lambda_1  + \frac{\eta_n}{n}\lambda_1^2 }} \bM(k,\ell) 
\\ & \geq \frac{c}{\lambda_1} \ \bM(k,\ell) 
\end{split}
\end{align}
for some $c > 0 $ and $n$ large enough since $\exp(-\eta_n(\lambda_1 - \lambda_2)) \rightarrow 0$.

For the first term of $L_q$, $q=3$ we have
\begin{align*}
    L_{3,1}^U&\leq \frac{1}{\sqrt{n}}\max_i \frac{\E(U_i^T \bigvar U_i)^{3/2}}{\|\bigvar\|_F^{3}}\\
    &\leq \frac{M_d^{3/2}}{\sqrt{n}}\frac{\E\|V_\perp^T(X_iX_i^T-\Sigma)v_1\|^3}{\|\bigvar\|_F^3}\qquad \mbox{ Since $\|\bigvar\|\leq M_d \eta_n$ from Eq~\ref{eq:bigvar}}\\
    &\leq C\frac{M_d^{3/2}\eta_n^3 \lambda_1^3}{\sqrt{n}}\E\left(\frac{\|V_\perp^TX_1X_1^T v_1\|}{\|\bM\|_F}\right)^3\\
\end{align*}

Similarly, for the Gaussian analog, we have that:
\begin{align*}
    L_{3,1}^Z&\leq \frac{1}{\sqrt{n}}\max_i \frac{\E(Z_i^T \bigvar Z_i)^{3/2}}{\|\bigvar\|_F^{3}}\\
    &\leq \frac{M_d^{3/2}\eta_n^{3/2}}{\sqrt{n}}\max_i\frac{\E\|Z_i\|^3}{\|\bigvar\|_F^3}\\
     &\leq \frac{M_d^{3/2}\eta_n^{3/2}}{\sqrt{n}}\frac{\E\|\tilde{Z}_i\|^3}{\|\bigvar\|_F^3}\\
    &\leq C\frac{M_d^{3/2}\eta_n^3 \lambda_1^3}{\sqrt{n}}\E\left(\frac{\|\tilde{Z}_1\|}{\|\bM\|_F}\right)^3\\
\end{align*}


%
For the second term, using the definition of $U_i$ in Eq~\ref{eq:Ui} we have:
\begin{align*}
    L_{3,2}^U&\leq \frac{1}{n}\max_{i<j} \frac{E|U_i^T U_j|^3}{\|\bigvar\|_F^{3}}\\
    &=\frac{1}{n}\max_{i<j} \frac{E|v_1^T (X_iX_i^T-\Sigma)D_{i+j-2}(X_jX_j^T-\Sigma)v_1|^3}{\|\bigvar\|_F^{3}}\\
    &\leq \frac{1}{n}\frac{\left(\E\|V_\perp^T(X_iX_i^T-\Sigma)v_1\|^3\right)^2}{\|\bigvar\|_F^3}\leq \frac{\eta_n^3 \lambda_1^3}{ n}\frac{\left(\E\|V_\perp^T(X_iX_i^T)v_1\|^3\right)^2}{\|\bM\|_F^3}\\
\end{align*}

For $K_3$, we have:
\begin{align*}
    K_3^{3} &= \frac{1}{n} \sum_{i=1}^n \E\left|\frac{U_i^TU_i - E(U_i^TU_i)}{f} \right|^3\\
    &\leq \max_i \frac{\E (U_i^T U_i)^3+(E U_i^T U_i)^3}{f^3}\leq 2\max_i \frac{\E (U_i^T U_i)^3}{\|\bigvar\|_F^3}\\
    &\leq 2\eta_n^3\lambda_1^{3}\frac{\E \|V_\perp^T(X_iX_i^T-\Sigma)v_1\|^6}{\|\bM\|_F^3}\\
\end{align*}
Finally, for $J_1$ we have:
\begin{align*}
     J_n &= \frac{\sum_{i=1}^n \var(U_i^T U_i)}{(nf)^2} \leq \frac{\sum_{i=1}^n\E (U_i^T U_i)^2}{n^2f^2} \\
     &\leq \frac{\eta_n^2\lambda_1^2}{n} \frac{\E[\|V_\perp(X_1X_1^T-\Sigma)v_1\|^4]}{\|\bM\|_F^2}
\end{align*}
%
The first makes $L_{3,2}$, $K_3^3/n$ and $J_n$ go to zero. The two conditions also imply $\frac{\E[\|V_\perp(X_1X_1^T-\Sigma)v_1\|^3]}{\|\bM\|_F^3}=o(\sqrt{n})$, which implies $L_{3,1}\rightarrow 0$.

Finally, we collect remainder terms  and show that their contribution to the inner product is negligible using anti-concentration. Observe that,
 \begin{align}
 \label{eq-kolomogorov-bound}
 \begin{split}
& \ \sup_{t \in \mathbb{R}} \left| P(n/\eta_n \sin^2(w,v) \leq t ) - P(Z_n^TZ_n \leq t) \right|
\\ = & \ \sup_{t \in \mathbb{R}} \left| P\left(R_1 \cdot \frac{(S_n + R_2 + R_3)^T(S_n + R_2 + R_3)}{ f }   \leq t \right) - P\left( \frac{Z_n^TZ_n}{f}  \leq t \right)  \right|
\end{split}
\end{align}
Now will will lower bound the above quantity.  Observe that  
\begin{align}
\begin{split}
\label{eq-clt-lower-bound-1}
 & \  P\left(R_1 \cdot \frac{(S_n + R_2 + R_3)^T(S_n + R_2 + R_3)}{ f }   \leq t \right) \\
 \geq & \  P\left( R_1 \cdot  \frac{S_n^TS_n}{f} \left(1 + \frac{2 \norm{R_2} + 2\norm{R_3}_2 }{\sqrt{S_n^TS_n}  }\right) + \frac{R_1 \cdot \norm{ R_2 +R_3}^2}{f}  \leq t \right) \\ 
 = & \  P\left(  R^\prime \cdot \frac{S_n^TS_n}{f}  + \widetilde{R} \leq t \right), \ \text{say}.
 \end{split}
\end{align}

Now, for $\delta_n= o(\sqrt{f})$,
we have that:%
\begin{align}
\label{eq:delta_denom-bound}
P\left( S_n^TS_n \leq \delta_n^2 \right) \leq \sup_{t \in \mathbb{R}} \left|P( S_n^TS_n \leq t) - P(Z_n^T Z_n \leq t)\right| + P(Z_n^T Z_n \leq \delta_n^2) \rightarrow 0
\end{align}
 Note that $\delta_n = o(1)$ suffices since $f$ is bounded away from zero under Eq \ref{eq:clt-lb} as shown in Eq \ref{eq-lower-bound-V-bar}.

Now, choose $\epsilon_n$ satisfying $\epsilon_n = o(1)$ $\epsilon_n = \omega\left(\sqrt{\frac{\eta_n^3 M_d^2 \log d}{n}} \right)$, define the set:
\begin{align*}
\mathcal{G} = \biggl\{\left|R^\prime - 1\right| \leq \epsilon_n, \  \bigl|\widetilde{R} \bigr| \leq \epsilon_n \biggr\}      
\end{align*}
so that $P(\mathcal{G}^c) \rightarrow 0$ with the choice of $\delta_n$ in Eq. \ref{eq:delta_denom-bound}.  
By using the fact that, for any two sets $A$ and $B$, $1 \geq P(A) + P(B) - P(A \cap B)$ and hence $P(A \cap B) \geq P(A) - P(B^c)$, we have that:
\begin{align}
\label{eq-clt-lower-bound-2}
\begin{split}
& \ P\left(  R^\prime \cdot \frac{S_n^TS_n}{f}  + \widetilde{R} \leq t \right) \\ 
= &  \ P\left( R^\prime \cdot  S_n^TS_n/f  + \widetilde{R} \leq   t \ \cap \ \mathcal{G}  \right) +  P\left( R^\prime \cdot  S_n^TS_n/f  + \widetilde{R}  \leq  t \ \cap \  \mathcal{G}^c  \right) \\ 
\geq & \ P\left( \frac{S_n^TS_n}{f} \leq \frac{t}{1+\epsilon_n} - \epsilon_n \right) - P(\mathcal{G}^c) 
\end{split}
\end{align}
Therefore, 
\begin{align}
\label{eq-clt-lower-bound-3}
\begin{split}
& \  P\left(\frac{n/\eta_n \sin^2(w,v)}{f} \leq t \right) - P\left(\frac{Z_n^TZ_n}{f} \leq t\right)  \\
\geq  & \   P\left( \frac{S_n^TS_n}{f} \leq \frac{t}{1+\epsilon_n} - \epsilon_n \right) - P\left( \frac{Z_n^TZ_n}{f} \leq  \frac{t}{1+\epsilon_n} - \epsilon_n  \right) \\ 
 & \ + P\left( \frac{Z_n^TZ_n}{f} \leq  \frac{t}{1+\epsilon_n} - \epsilon_n  \right) - P\left( \frac{Z_n^TZ_n}{f} \leq t \right) - P(\mathcal{G}^c) = I + II - III
 \end{split}
\end{align}
Now, we may upper bound $III \rightarrow 0$ arising from our choice of $\delta_n$, and $II$ goes to $0$ if the conditions of Proposition \ref{prop:clt-adaptation} are satisfied, and $I \rightarrow 0$ due to Proposition \ref{prop-chisq-ac}.

Now for the upper bound, since $\norm{R_i}_2 \geq 0$, observe that we may bound Eq \ref{eq-kolomogorov-bound} with:
 \begin{align*}
& \  P\left(R_1 \cdot \frac{(S_n + R_2 + R_3)^T(S_n + R_2 + R_3)}{ f }   \leq t \right) \\ 
 \leq & \  P\left(R_1 \cdot \frac{S_n^TS_n}{f} \left(1 -\frac{ 2\norm{ R_2} + 2\norm{ R_3}}{ \sqrt{S_n^TS_n}}\right) - \frac{R_1 \cdot \norm{R_2} \norm{R_3}  }{f}   \leq t \right)
\end{align*}
We may now lower bound the negative terms and arrive at an identical expression to the lower bound. The result follows. 
\end{proof}

With the central limit theorem in hand, we are now ready to give the proof for Corollary~\ref{cor:sinsq}.  
\bigskip 

\noindent \textit{Proof of Corollary~\ref{cor:sinsq}.}
Observe that the approximating distribution $Z_n^TZ_n$ has expectation $\mathrm{trace}(\bar{\mathbb{V}}_n)$ and variance $f =  \norm{\bar{\mathbb{V}}_n}_F$.  Therefore, for any $M >0$, it follows that:
\begin{align*}
&  \ P\left( \frac{n/\eta_n \sin^2(\hat{v}_1, v_1) - \mathrm{trace}(\bar{\mathbb{V}}_n)}{f} > M \right) \\ 
\leq  & \ \sup_{t \in \mathbb{R}} \left|P\left(n/\eta_n \sin^2(\hat{v}_1, v_1) > t \right) - P\left( Z_n^T Z_n > t \right) \right| + P\left(\frac{ Z_n^TZ_n -\mathrm{trace}(\bar{\mathbb{V}}_n)}{f}  > M \right)
\end{align*}
The first term goes to zero under the conditions of Theorem 1. Chebychev's inequality implies that there exists $M >0$ such that the latter probability can be made smaller than $\epsilon/2$ for any $\epsilon > 0$.  Hence,
\begin{align*}
\frac{n/\eta_n \sin^2(\hat{v}_1, v_1) - \mathrm{trace}(\bar{\mathbb{V}}_n)}{f} = O_P(1).    
\end{align*}
Therefore, under the conditions in Theorem 1,
\begin{align*}
\sin^2(\hat{v}_1, v_1) = \frac{\eta_n}{n} \left[ \mathrm{trace}(\bar{\mathbb{V}}_n) + O_P\left(  \norm{\bar{\mathbb{V}}_n}_F \right)\right]
\end{align*}
We now derive bounds for $\mathrm{trace}(\bar{\mathbb{V}}_n)$ and $\norm{\bar{\mathbb{V}}_n}_F$.     
Let  $\Lambda_\perp$ be a diagonal matrix with $\Lambda_\perp(i,i)=(1+\eta_n\lambda_{i+1}/n)/(1+\eta_n\lambda_1/n)$, $i=1,\dots,d-1$.  Recall that:
 \ba{
 \mathbb{M} :=  \mathbb{E}\left[V_\perp^T(X_1^Tv_1)^2 X_1X_1^T V_\perp \right ].
 }
\begin{align*}
    \bigvar 
   &= \frac{\eta_n}{n}V_\perp \left(\sum_i \Lambda_\perp^{i-1}\bM\Lambda_\perp^{i-1}\right)V_\perp^T
\end{align*}
So now observe that,
\bas{
\norm{\bigvar}_F&=\frac{\eta_n}{n}\norm{\sum_i \Lambda_\perp^{i-1}\bM\Lambda_\perp^{i-1}}_F\\
\trace(\bigvar)&= \frac{\eta_n}{n}\trace\bb{\sum_i \Lambda_\perp^{i-1}\bM\Lambda_\perp^{i-1}}
}
A direct calculation shows that the $k,\ell^{th}$ entry of the sum $\sum_i \Lambda_\perp^{i-1}\bM\Lambda_\perp^{i-1}$ is:

\begin{align}
\begin{split}
\label{eq-cov-mat-derivation}
&\sum_{i\geq 1}\left(\frac{1+\eta_n\lambda_{k+1}/n}{1+\eta_n\lambda_1/n}\right)^{i-1}\left(\frac{1+\eta_n\lambda_{\ell+1}/n}{1+\eta_n\lambda_1/n}\right)^{i-1}\bM(k,\ell)\\
&\leq\frac{n \bM(k,\ell)}{\eta_n}\frac{(1+\frac{\lambda_1\eta_n}{n})^2}{2\lambda_1-(\lambda_{k+1}+ \lambda_{k+1}) + \frac{\eta_n}{n}(\lambda_1^2 - \lambda_k\lambda_l) }\\
&\leq \frac{n}{\eta_n}\frac{C \bM(k,\ell)}{\lambda_1 - \lambda_2 }
\end{split}
\end{align}
for some $0 < C < \infty$.

Therefore, by Eq~\ref{eq:bigvar}, we have 
\bas{\trace(\bigvar) &\leq C\frac{\trace(\bM)}{\lambda_1-\lambda_2}\leq C\frac{M_d}{\lambda_1-\lambda_2}\\
\|\bigvar\|_F&\leq \frac{C\|\bM\|_F}{\lambda_1-\lambda_2}\leq C'\frac{M_d}{\lambda_1-\lambda_2}
}
The last step is true since:
\bas{
\trace(\bM)&=\trace(\mathbb{E}\left[V_\perp^T(X_1^Tv_1)^2 X_1X_1^T V_\perp \right ]) \\
&= \trace(\mathbb{E}\left[V_\perp^T(X_1X_1^T-\Sigma)v_1v_1^T (X_1X_1^T-\Sigma) V_\perp \right ])\\
& = \E\bb{\trace\left[V_\perp^T(X_1X_1^T-\Sigma)v_1v_1^T (X_1X_1^T-\Sigma) V_\perp \right ]}\\
& = \E\|V_\perp^T(X_1X_1^T-\Sigma)v_1\|^2 \leq  M_d
}
Similarly, 
\bas{
\norm{\bM}_F&=\left\|\mathbb{E}\left[V_\perp^T(X_1^Tv_1)^2 X_1X_1^T V_\perp \right ]\right\|_F \\
&=  \norm{\mathbb{E}\left[V_\perp^T(X_1X_1^T-\Sigma)v_1v_1^T (X_1X_1^T-\Sigma) V_\perp \right ]}_F\\
& \leq \E\|X_1X_1^T-\Sigma\|_{op}^2=  M_d    
}
where in the last line we used the fact that $\norm{xx^T }_{op} = \norm{xx^T}_F$ for $x \in \mathbb{R}^d$ since $xx^T$ is rank 1. \qed \\

\subsection{Adaptation of high-dimensional central limit theorem}\label{sec:CLTadapt}
Let $U_1, \ldots, U_n,$ be independent random vectors in $\mathbb{R}^p$ such that $E(U_i) = 0$ and $\mathrm{Var}(U_i) = \mathbb{V}_i$. Define a Gaussian analog of $Y_i$, denoted $Z_i$, which satisfies $E(Z_i) = 0$ and $\mathrm{Var}(Z_i) = \mathbb{V}_i$.  Furthermore, let $\bar{\mathbb{V}}_n = \frac{1}{n} \sum_{i=1}^n \mathbb{V}_i$, $g_i = \mathrm{Var}(U_i^T U_i)$, $f_1 = \mathrm{trace}(\bar{\mathbb{V}}_n)$, and $f = \norm{\bar{\mathbb{V}}_n}_F$.  For $0 < \delta \leq 1$, $q = 2 + \delta$, and $\beta \geq 2$ define the following quantities:
\begin{align*}
L_q^U &=          \frac{1}{n}\sum_{i=1}^n  \frac{E(U_i^T \bar{\mathbb{V}}_n U_i)^{q/2}}{n^{\delta/2} f^q}          +   \frac{1}{{n \choose 2}}\sum_{1\leq i<j 
\leq n}\frac{E(|U_i^TU_j|^q)}{n^\delta f^q  } \\ 
L_q^Z &= \frac{1}{n}\sum_{i=1}^n \frac{E(Z_i^T \bar{\mathbb{V}}_n Z_i)^{q/2}}{n^{\delta/2} f^q}  \\ 
K_\beta^{\beta} &= \frac{1}{n} \sum_{i=1}^n E\left|\frac{U_i^TU_i - E(U_i^TU_i)}{f} \right|^\beta \\ 
 J_n &= \frac{\sum_{i=1}^n g_i}{(nf)^2} 
\end{align*} 

The following proposition is an adaptation of ~\cite{10.1093/biomet/asz020}, which is stated for IID random variables, to independent but non-identically distributed random variables.  While the changes are minor, we provide a proof below detailing the adaptation for completeness.   

\begin{proposition}\label{prop:clt-adaptation}
Suppose that $L_q^U \rightarrow 0$, $L_q^Z \rightarrow 0$, $J_n \rightarrow 0$, $ n^{1-\beta}K_\beta^\beta \rightarrow 0$.  Then,
\begin{align*}
\sup_{t \in \mathbb{R}} \left|P\left(  n\bar{U}_n^T\bar{U}_n \leq t  \right) - P\left(  n\bar{Z}_n^T\bar{Z}_n \leq t  \right) \right| \rightarrow 0
\end{align*}
\end{proposition}
\begin{proof}
Since a Lindeberg argument is easier with diagonals removed, we will show that the removal of these terms is negligible.  Observe that:
\begin{align*}
  & \ \sup_{t \in \mathbb{R}} \left| P( n \bar{U}_n^T \bar{U}_n \leq t) - P(n \bar{Z}_n^T \bar{Z}_n \leq t)  \right |  \\ 
 \leq  & \ \sup_{t^\prime \in \mathbb{R}} \left| P\left(\frac{n \bar{U}_n^T \bar{U}_n - f_1}{f} \leq t^\prime  \right) - P\left( \frac{\sum_{i\neq j} U_i^T U_j}{nf} \leq t^\prime \right) \right|
\\ + & \ \sup_{t^\prime \in \mathbb{R}} \left| P\left( \frac{\sum_{i\neq j} U_i^T U_j }{nf} \leq t^\prime \right) - P\left(\frac{\sum_{i\neq j} Z_i^T Z_j}{nf} \leq t^\prime \right) \right|
\\ + &  \  \sup_{t^\prime \in \mathbb{R}} \left| P\left(\frac{\sum_{i\neq j} Z_i^T Z_j }{nf} \leq t^\prime \right)- P\left( \frac{n \bar{Z}_n^T \bar{Z}_n - f_1}{f} \leq t^\prime \right)\right| 
\\ =& \ I +II + III, \  \text{say}.
\end{align*}
We will start by bounding $III$. First note that $\frac{1}{\sqrt{n}} \sum_{i=1}^n Z_i \sim  \mathcal{N}(0, \bar{\mathbb{V}}_n)$.  Let $\bar{\mathbb{V}}_n =  Q^TD Q$ denote the eigendecomposition, with diagonal entries of $D$ given by $\lambda_1 \geq  \ldots \geq  \lambda_d$ and let $g \sim \mathcal{N}(0, \mathrm{I}_d)$. It follows that:
\begin{align*}
n \bar{Z}_n^T\bar{Z}_n & \stackrel{d}{=} (Q D^{1/2} Q^T g)^T (Q D^{1/2} Q^Tg)
\\ &\stackrel{d}{=} g^T D g 
\end{align*}
Notice that $V :=  g^T D g \sim \sum_{r=1}^d \lambda_r \eta_r$, where $\eta_1, \ldots, \eta_d \sim \chi^2(1)$.  Now define $R_n^Z = \frac{\frac{1}{n}\sum_{i=1}^nZ_i^T Z_i- f_1}{f}$.
Notice that:
\begin{align}
\label{eq:z-anticoncentration}
\begin{split}
& \  P\left( \frac{n \bar{Z}_n^T\bar{Z}_n - f_1 }{f} \leq t \right) - P\left( \frac{\sum_{i \neq j} Z_i^T Z_j}{f} \leq t \right)
\\ = & \ P\left( \frac{n \bar{Z}_n^T\bar{Z}_n - f_1 }{f} \leq t \right) -  P\left( \frac{n \bar{Z}_n^T\bar{Z}_n - f_1 }{f} - R_n^Z \leq t \right)
\\  \leq & \  P(t^\prime \leq V \leq  t^\prime+h_n) + P( |R_n^Z| > h_n)
\end{split}
\end{align}
Under the conditions $J_n \rightarrow 0$, $n^{1-\beta}K_\beta^\beta \rightarrow 0$, Nagaev's inequality implies that one may choose $h_n \rightarrow 0$ such that $P( |R_n^Z| > h_n) \rightarrow 0$.  The desired anti-concentration for the first term in the previous display follows from Lemma S2 of \cite{10.1093/biomet/asz020}.  We may also derive the lower bound $ P(t^\prime \leq V \leq  t^\prime+h_n) - P( |R_n^Z| > h_n)$ in a similar manner.

To adapt $II$, consider the smoothed indicator function:
\begin{align*}
g_{\psi,t}(x) = \left[1-\min \{1, \max(x-t, 0) \}^4\right]^4.
\end{align*}
This function satisfies:
\begin{align*} & \max_{x,t} \{ |g_{\psi,t}^\prime(x)| + |g_{\psi,t}^{\prime \prime}(x)| + |g_{\psi,t}^{\prime \prime \prime}(x)| \} < \infty \\ 
& \mathbbm{1}_{x \leq t} \leq g_{\psi,t} \leq \mathbbm{1}_{x \leq t +\psi^{-1}}.
\end{align*}
Therefore, we may bound the approximation error with smoothed indicator function by again using anti-concentration of the weighted $\chi^2$.   In what follows, let:
\begin{align*}
S_n^U = \frac{1}{nf}\sum_{i \neq j} U_i^T U_j, \ \ \ S_n^Z = \frac{1}{nf}\sum_{i \neq j} Z_i^T Z_j
\end{align*}
We have that:
\begin{align*}
\begin{split}
\ & \  P(S_n^U \leq t) - P(S_n^Z \leq t)  \\ 
\leq & \ P(S_n^U \leq t) - P(S_n^Z \leq t + \psi^{-1}) +  P(S_n^Z \leq t+ \psi^{-1}) - P(S_n^Z \leq t) \\ 
\leq & \  Eg_{\psi,t}(S_n^U) -Eg_{\psi,t}(S_n^Z) + III + P(t \leq V \leq t+ \psi^{-1}). 
\end{split}
\end{align*}
An analogous argument establishes a lower bound of $g_{\psi,t}(S_n^U) -Eg_{\psi,t}(S_n^Z) - III - P(t-\psi^{-1} \leq V \leq t)$. Choosing $\psi_n \rightarrow \infty$, the last term goes to zero. A Lindeberg telescoping sum argument leads to the following bound for the leading term:
\begin{align*}
 \left|Eg_{\psi,t}(S_n^U) -Eg_{\psi,t}(S_n^Z)\right| \leq \sum_{i=1}^n c_q (E|\Delta_i|^q + E|\Gamma_i|^q), 
\end{align*}
where:
\begin{align*}
H_i = \sum_{j=1}^{i=1} U_i + \sum_{j=i+1}^{n} Z_i, \ \  \Delta_i = \frac{U_i^TH_i}{nf}, \ \ \ \Gamma_i =  \frac{Z_i^TH_i}{nf}.   
\end{align*}
We may use analogous reasoning to bound these terms. Let $\xi \sim N(0,1)$. Conditioning on $U_1 = u_i$, by Rosenthal's inequality:
\begin{align}
\label{eq:delta-bound}
\begin{split}
 \mathbb{E}\left[\bigr|\Delta_i\bigr|^q \ \rvert \  U_i \right] \leq & \ \sum_{j=1}^{i-1} \frac{\mathbb{E}[|U_j^T u_i|^q]}{n^q f^q} + \sum_{j=i+1}^n\frac{\mathbb{E}[|Z_j^Tu_i|^q]}{n^q f^q} + n^{q/2} \frac{\left(u_i^T\bar{\mathbb{V}}_n u_i \right)^{q/2}}{n^q f^q} \\  
\leq & \ \sum_{j=1}^{i-1} \frac{\mathbb{E}[|U_j^T u_i|^q]}{n^q f^q} + \sum_{j=i+1}^n \|\xi\|_q^q \frac{\left(u_i^T \mathbb{V}_j u_i \right)^{q/2}}{n^qf^q}  + \frac{\left(u_i^T\bar{\mathbb{V}}_n u_i \right)^{q/2}}{n^{q/2} f^q} 
\end{split}
\end{align}
Taking expectations, it follows that: 
\begin{align*}
\sum_{i=1}^n\mathbb{E}\left[\bigr|\Delta_i\bigr|^q \right] \lesssim \frac{1}{{n \choose 2}} \sum_{1 \leq i < j \leq n} \frac{\mathbb{E}\left[|U_i^TU_j|^q \right]}{n^{\delta}f^q} +  \frac{1}{n} \sum_{i=1}^n \frac{\mathbb{E}\left|U_i^T\bar{\mathbb{V}}_n U_i \right|^{q/2}}{n^{\delta/2}f^q}  
\end{align*}
 Now, for $\Gamma_i$, we may use Rosenthal's inequality so that:
\begin{align*}
\sum_{i=1}^n\mathbb{E}\left[\bigr|\Gamma_i\bigr|^q \right] \leq \frac{1}{n} \sum_{i=1}^n \frac{\mathbb{E}\left|U_i^T\bar{\mathbb{V}}_n U_i \right|^{q/2}}{n^{\delta}\delta f^q} + \frac{1}{n} \sum_{i=1}^n \frac{\mathbb{E}\left[\left|Z_i^T\bar{\mathbb{V}}_n Z_i \right|^{q/2}\right]}{n^{\delta}\delta f^q} + \frac{1}{n} \sum_{i=1}^n \frac{\mathbb{E}\left(Z_i^T\bar{\mathbb{V}}_n Z_i] \right)^{q/2}}{n^{q/2} f^q} 
\end{align*}

While omitted in the original proof, in the IID case, the latter terms may be bounded by using an eigendecomposition along with properties of the Gaussian.  However, since the $Z_i$ do not have variance matrix $\mathbb{V}_n$, we instead oppose the additional condition for $L_q^Z$.   By the assumptions made in theorem, it follows that $II \rightarrow 0$.

Finally, for $I$, we have that:
\begin{align*}
 \ & \ P\left(\frac{n \bar{U}_n^T \bar{U}_n - f_1}{f} \leq t  \right) - P\left( \frac{\sum_{i\neq j} U_i^T U_j}{nf} \leq t \right) \\
 \leq & \  P(S_n^X \leq t+ h_n ) -  P(S_n^U \leq t+ h_n) + P(|R_n^X| > h_n) 
  \\  & \  +   P( t \leq V \leq t+ 2h_n) + P(|S_n^Z| > h_n) 
\end{align*}
 Using bounds from $II$ and $III$ along with anti-concentration properties, we may conclude that $I \rightarrow 0$.


%

\end{proof}

\subsection{Supporting lemmas for CLT}\label{sec:CLT-support}
In several of our lemmas, we use the following technique from \cite{jain2016streaming} that facilitates analysis for initializations from a uniform distribution on $\mathcal{S}^{d-1}$ particularly when $d$ is large. 

\begin{proposition}[Trace trick]\label{prop:trace}
Suppose that $u$ is drawn from a uniform distribution on $\mathcal{S}^{d-1}$.  Then, for any $A \in \mathbb{R}^{d \times d}$ and  $v \in \mathbb{R}^d$ satisfying $\norm{v}=1$, with probability at least $1- C\delta$, for some $C>0$ independent of $A$ and $0< \delta < 1$, 
\begin{align*}
\frac{u^TA^TAu}{(v^Tu)^2} \leq \frac{\log(1/\delta) \ \mathrm{trace}(AA^T) }{\delta^2 }    
\end{align*}
\end{proposition}
\begin{proof}
First, we recall the well-known fact that $u= g/\norm{g}$, where $g \sim N(0, I_d)$. Therefore, $\norm{g}$ cancels as follows:
\begin{align*}
\frac{u^TA^TAu}{(v^Tu)^2} =  \frac{g^TA^TAg}{(v^Tg)^2}
\end{align*}

Furthermore, observe that $g^TA^TAg$ may be viewed as a weighted sum of independent $\chi^2(1)$ random variables.  In particular, by an eigendecomposition argument, for $\eta_1, \ldots \eta_r \sim \chi^2(1)$ and $A = VDV^T$, 
\begin{align*}
g^T(VDV^T)(VDV^T)g &= g^TVD^2 V^Tg \\ 
&\stackrel{d}{=} g^TD^2g 
\\ &= \sum_{r=1}^p \lambda_r^2 \eta_r = \psi, \text{ say}
\end{align*} 
where above we used the fact that $V^Tg \sim N(0,I_d)$.  Now observe that $\mathbb{E}[\psi] = \sum_{r=1}^p \lambda_r^2 =  \norm{A}_F^2$ and that $\eta_r$ is sub-Exponential.  Therefore, by by Bernstein's inequality (see for example Theorem 2.8.2 of \cite{verhsynin-high-dimprob}), for some $K >0$, $C_1 > 0$, $0 < \delta < 1$, 
\begin{align*}
P\left( \psi - \mathbb{E}[\psi] > (\log(1/\delta)-1) \norm{A}_F^2 \right) &\leq     \exp\left\{- \min\left(\frac{\log^2(1/\delta) \norm{A}_{\mathcal{S}_2}^4}{4K^2\norm{A}_{\mathcal{S}_4}^4} , \frac{\log(1/\delta) \norm{A}_{\mathcal{S}_2}^2}{2K\norm{A}_{\mathcal{S}_\infty}^2} \right)  \right\}\\ 
& \leq   \exp\left\{- \min\left(\frac{\log^2(1/\delta) }{4K^2}, \frac{\log(1/\delta)}{2K} \right) \right\} \leq C_1 \delta
\end{align*}
where above $\norm{\cdot}_{\mathcal{S}_p}$ is the $p$th Schatten-Norm, defined as $(\sum_{r=1}^d s_r^p)^{1/p}$, where $s_r$ is the $r$th singular value and satisfies $\norm{\cdot}_{\mathcal{S}_q} \leq \norm{\cdot}_{\mathcal{S}_p}$ for $p \leq q$.  
Now for the denominator, since $v^Tg \sim N(0,1)$ and $(v^Tg)^2 \sim \chi^2(1)$, Proposition \ref{prop-chisq-ac} yields:
\begin{align*}
P( (v^Tg)^2 \leq \delta^2) \leq \frac{2\delta}{\sqrt{\pi}} 
\end{align*} 
The result follows. 
\end{proof}
The following anti-concentration result for weighted $\chi^2$ distributions is also used in several places. 
\begin{proposition}[Weighted $\chi^2$ anti-concentration, \cite{10.1093/biomet/asz020}] 
\label{prop-chisq-ac}
Let $a_1 \geq \cdots \geq a_p \geq 0$ such that $\sum_{r=1}^p a_i^2 =1$ and suppose that $\xi_1, \ldots, \xi_p \sim \chi^2(1)$. Then,
\begin{align*}
\sup_{t \in \mathbb{R}} P\left( t \leq \sum_{r=1}^p a_r \xi_r \leq t+h \right) \leq \sqrt{\frac{4h}{\pi}}
\end{align*}
\end{proposition}

We now present a concentration result for matrix products that follow immediately from  Corollary 5.4 of ~\cite{huang2020matrix}. 
\begin{lemma}[Expectation bounds for operator norms of of matrix products]
\label{lem:ward}

Let $\bc_k=\prod_{j=1}^k (I+\eta_n X_jX_j^T/n)$.
We have,
\ba{\label{eq:ward-op-err}
\E\|\bc_k-\E \bc_k\|^2 \leq  \frac{M_d e
\eta_n^2(1+2\log d)k}{n^2}(1+\eta_n\lambda_1/n)^{2k}.
}
For the expectation, we have,  if $\frac{(1 + 2 \log d)M_d\eta_n^2}{n} \leq 1$:
\ba{\label{eq:ward-op-exp}
\E\|\bc_k\|^2\leq \exp\bb{2\sqrt{2M_d\frac{k\eta_n^2}{n^2}\bb{2M_d\frac{k\eta_n^2}{n^2} \vee \log d}}}\bb{1+\eta_n\lambda_1/n}^{2k}.
}
\end{lemma}
\begin{proof}
 We invoke Corollary 5.4 in ~\cite{huang2020matrix}  with $\|\E (I+\eta_n/n X_iX_i^T)\|\leq 1+\eta_n\lambda_1/n$,  $\sigma_i^2=M_d\frac{\eta_n^2}{n^2}$, and $\nu=M_d\frac{k\eta_n^2}{n^2}$. Note that for a random matrix $M$ with Schatten norm $\|M\|_{\mathcal{S}_p}$,
$\E\|M\|\leq \sqrt{\E\|M\|_{\mathcal{S}_p}^2} $ and hence the same argument as in their proof invoking Eq~5.5 and 5.6 works.
\end{proof}



\begin{lemma}[Concentration of the norm for the CLT]
\label{lemma-norm-concentration}
For some $C >0$, and any $\epsilon >0$, $ 0 < \delta < 1$,
\begin{align*}
& \ P\bb{\left|\frac{\|B_nu_0\|}{|a_1|(1+\eta_n\lambda_1/n)^{n}}-1\right|\geq \epsilon }
\\ & \leq  \frac{ d\exp\left(-\eta_n(\lambda_1 - \lambda_2) + \frac{\eta_n^2}{n} (\lambda_1^2 + M_d) \right) + \frac{\eta_n^2}{n} M_d \exp\left( \frac{\eta_n^2}{n}  \right) }{ 4\log^{-1}(1/\delta) \delta^2 \epsilon^2\left(1+\frac{\eta_n^2 \lambda_1^2}{n} \right) } + \frac{e^2 \eta_n^2 M_d(1+ \log d) }{n\epsilon ^2} + C \delta
\end{align*}
\end{lemma}
\begin{proof}
Consider the bound:
\begin{align*}
\left|\frac{\|B_nu_0\|}{|a_1|(1+\eta_n\lambda_1/n)^{n}}-1\right|  & \leq  \left|\frac{\|B_n v_1 a_1\| - \|a_1 T_0 v_1\|}{|a_1|(1+\eta_n\lambda_1/n)^{n}} \right| +   \frac{\|B_n V_\perp (V_\perp^T u_0) \|}{|a_1|(1+\eta_n\lambda_1/n)^{n}}  
\end{align*}
We will start by bounding the second term. 

Using Proposition~\ref{prop:trace}, 
observe that, with probability at least $1 - C \delta$, 
\begin{align*}
  & \ \frac{\|(B_n V_\perp V_\perp^Tg\|^2}{|v_1^Tg|^2(1+\eta_n\lambda_1/n)^{2n}} \leq  \  \frac{\log(1/\delta) \mathrm{trace}(V_\perp B_n B_n V_\perp^T) )}{\delta^2  (1+\eta_n\lambda_1/n)^{2n}} \\ 
\end{align*}
Let $\mathcal{G}$ denote the good set for which the upper bound above holds. Markov's inequality on the good set, together with Lemma 5.2 of \cite{jain2016streaming} with $\mathcal{V}_n \leq M_d$ yields that:
\begin{align*}
& \ P\left(\frac{\|B_n V_\perp V_\perp^Tg\|}{(1+\eta_n\lambda_1/n)^{n}} \geq \epsilon/2 \ \cap \ \mathcal{G} \right) \\ & \leq \frac{ d\exp\left(-\eta_n(\lambda_1 - \lambda_2) + \frac{\eta_n^2}{n} (\lambda_1^2 + M_d) \right) + \frac{\eta_n^2}{n} M_d \exp\left( \frac{\eta_n^2}{n}  \right) }{4\delta^2 \log^{-1}(1/\delta) \ \epsilon^2 \left(1+\frac{\eta_n^2 \lambda_1^2}{n} \right) }     
\end{align*}

 


Now we will bound the first summand.
By Lemma~\ref{lem:ward} Eq~\ref{eq:ward-op-err}, we have by Markov's inequality, 
\begin{align*}
& \ P\left( \frac{\|(B_n - T_0)\|_{op}}{ (1+\eta_n\lambda_1/n)^{n}}  > \epsilon/2 \right) \leq   \ \frac{e^2 M_d(1+ \log d) }{n\epsilon^2} 
\end{align*}

Combining the two bounds and the probability of $\mathcal{G}^c$, the result follows. 

\end{proof}

\begin{lemma}[Negligibility of  $V_\perp$ for the CLT]
\label{lem:CLTVperp}
Let $V_\perp$ denote the matrix of eigenvectors orthogonal to $v_1$. Also let $\lambda_i$ denote the $i^{th}$ largest eigenvalue of $\Sigma$. For some $C >0$, and any $\epsilon >0$, $ 0 < \delta < 1$, 
\begin{align*}
& \  P\left(\sqrt{\frac{n}{\eta_n}} \  \frac{\norm{V_\perp V_\perp^T B_n V_\perp V_\perp^T u_0}}{|a_1|(1+\frac{\eta_n \lambda_1}{n})^n} \geq \epsilon \right) \\
\leq & \  \frac{nd \log(1/\delta) \exp\bigl\{-2\eta_n(\lambda_1-\lambda_2)+\eta_n^2(\lambda_1^2+M_d)/n\bigr\} }{\eta_n \epsilon^2 \delta^{2}} +\frac{eM_d^2(1+2\log d)\eta_n^2 \epsilon^{-2} \log(1/\delta) \delta^{-2}}{n2(\lambda_1-\lambda_2)+\eta_n^2(\lambda_1^2-\lambda_2^2-M_d)} + C \delta   
\end{align*}
\end{lemma}


\begin{proof}

We consider bounding the squared quantity.  
We have, with probability at least $1 - C\delta$, using Proposition~\ref{prop:trace}, this quantity is upper bounded by:
\begin{align*}
 & \  \frac{\norm{(V_\perp V_\perp^T B_n V_\perp V_\perp)g}^2} { (v_1^T g)^2 (1+\eta_n\lambda_1/n)^{2n}}\\
\leq   & \   \frac{ \trace\bb{(V_\perp V_\perp^T B_n V_\perp V_\perp^T)(V_\perp V_\perp^T B_n V_\perp V_\perp^T)^T}}{\delta_n (v_1^T g)^2 (1+\eta_n\lambda_1/n)^{2n}}\\
 = & \   \frac{\trace\bb{V_\perp^TB_n V_\perp V_\perp^T B_n V_\perp }}{\delta_n^{3}(1+\eta_n\lambda_1/n)^{2n}} 
\end{align*}
Now we will bound the expectation of the numerator.

We will denote $\eta=\frac{\eta_n}{n}$ for simplicity. Let $U_i=I+\eta X_iX_i^T$ and $Y_i=X_iX_i^T-\Sigma$. We have that:
\begin{align}
    \alpha_n& :=\E\matprod{ B_n\vorth\vorth^T B_n^T}{\vorth\vorth^T }\notag\\
    &=\E\matprod{B_{n-1} \vorth\vorth^T B_{n-1}^{T}}{U_n \vorth\vorth^T U_n^T}\notag\\
    &=\matprod{\E B_{n-1} \vorth\vorth^TB_{n-1}^{T}}{\E U_n \vorth\vorth^T U_n^T}\label{eq:bnrec2}
\end{align}
Now we have:
\ba{\label{eq:esecond2}
\E U_n \vorth\vorth^T U_n^T&=\E \bb{I+\eta \Sigma}\vorth\vorth^T\bb{I+\eta \Sigma}^T+\eta^2\E Y_n \vorth\vorth^T Y_n^T\notag\\
&\preceq (1+2\eta\lambda_2+\lambda_2^2\eta^2)\vorth\vorth^T+\eta^2 M_d (\vorth\vorth^T+v_1v_1^T)\notag\\
&\preceq (1+2\eta\lambda_2+\lambda_2^2\eta^2+\eta^2M_d^2)\vorth\vorth^T+\eta^2 M_d v_1v_1^T
}



Finally, using Eqs~\ref{eq:bnrec2} and~\ref{eq:esecond2}, we have:
\ba{\label{eq:alphanv2}
\alpha_n
\leq \bb{1+2\eta\lambda_2+\eta^2(\lambda_2^2+M_d)}\alpha_{n-1}+\eta^2M_d\matprod{\E B_{n-1} \vorth\vorth^TB_{n-1}^{T}}{v_1v_1^T}
}
We will use the fact that,
\bas{
\langle (I+\eta\Sigma)^{n-1}V_\perp V_\perp^T (I+\eta\Sigma)^{n-1},v_1v_1^T\rangle&=0.
}

Thus, for some $N$ such that the condition $\eta_n^2 M_d \bk (1+2\log d)/n\leq 1$  holds for all rows of the triangular array with index $n > N$, we have by \bk Lemma \ref{lem:ward},
\bas{
&\matprod{\E B_{n-1} \vorth\vorth^TB_{n-1}^{T}}{v_1v_1^T}\\
&=\matprod{\E (B_{n-1} - (I+\eta\Sigma)^{n-1}) \vorth\vorth^T (B_{n-1} - (I+\eta\Sigma)^{n-1})^T}{v_1v_1^T}\\
&\leq \|\E (B_{n-1} - (I+\eta\Sigma)^{n-1}) \vorth\vorth^T (B_{n-1} - (I+\eta\Sigma)^{n-1})^T\| \\
&\leq \E \|B_{n-1} - (I+\eta\Sigma)^{n-1}\|^2 \\
&\leq  M_d e
\eta^2n(1+2\log d)(1+\eta_n\lambda_1/n)^{2(n-1)}.
}

Thus, Eq~\ref{eq:alphanv2} gives:
\begin{align*}
\alpha_n
&\leq \underbrace{\bb{1+2\eta\lambda_2+\eta^2(\lambda_2^2+M_d)}}_{c_1}\alpha_{n-1}+\eta^4 M_d^2  e
(1+2\log d)\underbrace{(n-1)(1+\eta\lambda_1)^{2(n-1)}}_{(n-1)c_2^{n-1}}\\
&=c_1\alpha_{n-1}+\eta^4 M_d^2  e
(1+2\log d) (n-1)c_2^{n-1}\\
&=c_1^n \alpha_0 +\eta^4 M_d^2  e
(1+2\log d) \sum_i c_1^{i-1}(n-i)c_2^{n-i}\\
&\leq c_2^n \bb{d(c_1/c_2)^n+\frac{eM_d^2(1+2\log d)\eta^4 n }{c_2-c_1}}\\
& \leq  (1+\eta_n\lambda_1/n)^{2n} \biggl( d (1-\lambda_1^2\eta_n^2/n)  \exp\{-2\eta_n(\lambda_1-\lambda_2)+\eta_n^2 (\lambda_1^2+M_d)/n\}  \\ 
&  \ \ \ \ \ \ \ \ \  \ \ \ \ \  \ \  +  \frac{eM_d^2(1+2\log d)\eta_n^3/n^2}{2(\lambda_1-\lambda_2)+\eta_n^2/n(\lambda_1^2-\lambda_2^2-M_d)}\biggr)
\end{align*}
\bk

where above we used the fact $e^x(1-\frac{x^2}{n}) \leq (1 + \frac{x}{n})^n \leq e^x$ for $|x| \leq n$ to bound $(c_1/c_n)^n$.


\end{proof}

\begin{lemma}[Negligibility of higher-order Hoeffding projections for the CLT]
\label{lemma-clt-hayek}
Let $ \beta_n =  \frac{\eta_n^2M_d}{n}$ and suppose that $0 \leq \beta_n \leq 1$. Then, for some $C>0$ and any $\epsilon >0$,  
\begin{align*}
P\left( \frac{\sqrt{\frac{n}{\eta_n}}  \norm{V_{\perp} V_{\perp}^T \sum_{k >1} T_k v_1}}{(1+ \frac{\eta_n\lambda_1}{n})^{n}} > \epsilon \right)  \leq  \frac{C\beta_n \eta_n }{(1-\beta_n)\epsilon^2}
\end{align*}
\end{lemma}
\begin{proof}
By Markov's inequality, it follows that: 
\begin{align*}
P\left( \frac{\frac{\sqrt{n}}{\eta_n}  \norm{V_{\perp} V_{\perp}^T \sum_{k >1} T_k v_1}}{(1+ \frac{\eta_n\lambda_1}{n})^{n}} > \epsilon \right) 
 & \leq \frac{\frac{n}{\eta_n^2}  \mathbb{E}\left[\norm{V_{\perp} V_{\perp}^T \sum_{k >1} T_k v_1}^2\right] }{\epsilon^2(1+ \frac{\eta_n\lambda_1}{n})^{2n}}   
\end{align*}
Now, by submultiplicativity of the operator norm and the fact that $\mathbb{E}[(P_{S_1}T)^T (P_{S_2})T] = 0$ for any two Hayek projections, the numerator is upper bounded by:

\begin{align*}
\left(\frac{n}{\eta_n} \right) \sum_{k=2}^n \left(\frac{\eta_n}{n}\right)^{2k} \sum_{|S| =k} \mathbb{E}\left[(v^\prime A_S u_0)^2 \right] 
 &\leq \left(\frac{n}{\eta_n} \right) \sum_{k=2}^n \sum_{|S|=k} \left(\frac{\eta_n}{n}\right)^{2k}  \mathbb{E}\left[\norm{A_S}_{op}^2\right] %
 \\  &  \leq \left(\frac{n}{\eta_n} \right) \sum_{k=2}^n \left(\frac{\eta_n}{n}\right)^{2k} \sum_{|S| = k}  \left(1 + \frac{\eta_n \lambda_{1}}{n}\right)^{2(n-k)} M_d^{k}
 \\ & \leq \eta_n M_d \left(1 + \frac{\eta_n \lambda_{1}}{n}\right)^{2n} \sum_{k=2}^n   \left(\frac{M_d\eta_n^2}{n}\right)^{k-1}
 \\ & \leq \left(1 + \frac{\eta_n \lambda_{1}}{n}\right)^{2n} \frac{\beta_n \eta_nM_d}{1-\beta_n}
\end{align*}
The result follows. 

\end{proof}

\section{Consistency of the online bootstrap}\label{sec:C}
In this section, we provide the detailed proof of Bootstrap consistency, i.e Theorem~2. 
\subsection{Proof of bootstrap consistency}\label{sec:bootproof}
\begin{proof}[Proof of Theorem~\ref{thm-oja-boot}]
Similar to the CLT, we will establish the negligibility of remainder terms and then use anti-concentration terms to argue that the contribution to the Kolmogorov distance is small. We then show that the bootstrap covariance of the main term approaches the weighted $\chi^2$ approximation in Theorem 1 with high probability.
Let $\widehat{v}_1$ denote the leading eigenvector estimated from Oja's algorithm and let $\widehat{V}_\perp$ denote its orthogonal complement.  Again, we have that:
 \begin{align*}
\frac{n}{\eta_n} \sin^2(v_1^*, \hat{v}_1)  = & \  \frac{n}{\eta_n}\frac{(B_n^*u_0)^T \widehat{V}_\perp\widehat{V}_\perp^T(B_n^* u_0)}{\norm{B_n^* u_0}^2} 
\\ = & \ \frac{(\sqrt{n/\eta_n}\widehat{V}_\perp\widehat{V}_\perp^TB_n^* u_0)^T (\sqrt{n/\eta_n} \widehat{V}_\perp\widehat{V}_\perp^TB_n^* u_0)}{\norm{B_n^* u_0}^2}
\end{align*}

We aim to show that the bootstrap distribution conditional on the data is close to the weighted $\chi^2$ approximation with high probability; therefore we may work the good set $\mathcal{A}_n$.  With the a slight abuse of notation, in the remainder terms below,  $O_P$ will be on the measure restricted to $\mathcal{A}_n$.

We first approximate the norm using Lemma \ref{lem:norm-concentration-lemma-boot}.  Analogous to the CLT, the corresponding remainder is given by:
\begin{align*}
R_1^* = \frac{\norm{B_n^* u_0}^2}{a_1^2(1 + \frac{\eta_n}{n}  \lambda_1)^{2n}} = 1 - O_{P}\left(\sqrt{d} \ \mathrm{exp} \left(-\frac{\eta_n}{2}(\lambda_1-\lambda_2) \right) +   \sqrt{\frac{\eta_n^2 M_d \log d}{n}} + \frac{\eta_n \alpha_n}{\sqrt{n}} \right)
\end{align*}
Next, we bound the contribution of the higher-order Hoeffding projections.  This step is different from the CLT in the sense that we handle both $v_1$ and $V_\perp$, using the fact that on the good set, even the Frobenius norm of certain terms are well-behaved.  By Lemma \ref{lemma-boot-hayek-remainder} we have that:
\begin{align*}
R_3^* := \sqrt{\frac{n}{\eta_n}} \cdot \frac{\widehat{V}_{\perp} \widehat{V}_{\perp}^T (B_n^*-T_1^*)u_0}{|a_1|(1+\eta_n/n \lambda_1)^n} = O_{P}\left(\exp\left(\sqrt{\frac{C M_d^2 \eta_n^2 \log d}{n}}  \right) \sqrt{\frac{\alpha_n^4 \eta_n^3}{n}}  \right) 
\end{align*}
Next, we bound the contribution of $V_\perp$ to the H\'{a}jek projection using Lemma~\ref{lem:hajekremboot}, as long as $\lambda_1M_d(\log d)^2\frac{\eta_n^2}{n}\rightarrow 0$,
\begin{align*}
R_2^* =  \ \sqrt{\frac{n}{\eta_n}} \cdot \frac{\widehat{V}_{\perp}\widehat{V}_{\perp}^T T_1^* V_\perp V_\perp^Tu_0}{|a_1|(1+\eta_n/n \lambda_1)^n}  = \  O_{P} \bb{\sqrt{\frac{\alpha_nM_d\eta_n^2}{n(\lambda_1-\lambda_2)}}}
\end{align*}
The final remainder term arises from the disparity between the orthogonal complements and the residuals of matrix products from their expectation. By Lemma~\ref{lem:boot-res}, with $\Delta_i=X_iX_i^T-X_{i-1}X_{i-1}^T$,
\begin{align*}
R_4^* = \sqrt{\frac{n}{\eta_n}}\left\|\frac{\vp\vp^T T_1^* v_1(v_1^T u_0)}{|v_1^T u_0|(1+\eta_n\lambda_1/n)^n} - \frac{\eta_n}{n}\sum_i W_i D_{i-1}\Delta_i v_1\right\| = O_P\bb{\sqrt{\frac{M_d\alpha_n\eta_n^3\log d}{n}}}
\end{align*}

Now, define:
\begin{align*}
S_n^* = \sqrt{\frac{n}{\eta_n}} \  \frac{V_\perp V_\perp^T T_1^* v_1}{(1+\frac{\eta_n \lambda_1}{n})^{n}}   
\end{align*}

Consider the following bound:
 \begin{align}
 \begin{split}
& \  P\left\{ \sup_{t \in \mathbb{R}} \left| P^*(n/\eta_n \sin^2(v_1^*, \widehat{v}_1) \leq t ) - P(Z^TZ \leq t) \right| > \epsilon \right\}
\\ = & \ P_{\mathcal{A}_n}\left\{ \sup_{t \in \mathbb{R}} \left| P^*\left(R_1^* \cdot \frac{(S_n^* + R_2^* + R_3^* + R_4^* )^T(S_n^* + R_2^* + R_3^* + R_4^*)}{ f }   \leq t \right) - P\left( \frac{Z^TZ}{f}  \leq t \right)  \right| > \epsilon  \right\}
\\ + & P_{\mathcal{A}_n^c}\left\{ \sup_{t \in \mathbb{R}} \left| P^*\left(R_1^* \cdot \frac{(S_n^* + R_2^* + R_3^* + R_4^* )^T(S_n^* + R_2^* + R_3^* + R_4^*)}{ f }   \leq t \right) - P\left( \frac{Z^TZ}{f}  \leq t \right)  \right| > \epsilon  \right\}
\end{split}
\end{align}
The second term is easily upper-bounded by $P(\mathcal{A}_n^c) \rightarrow 0 $, so we will bound the first term. To lower bound the Kolmogorov metric, we may follow the same reasoning used in Eqs \ref{eq-clt-lower-bound-1}, \ref{eq-clt-lower-bound-2}, \ref{eq-clt-lower-bound-3}, to deduce, on the good set $\mathcal{A}_n$, we have the lower bound:
 \begin{align*}
& \ P^*\left( \frac{S_n^{*T}S_n^*}{f} \leq  \frac{t}{1+\epsilon_n} - \epsilon_n  \right) -  P\left( \frac{Z^TZ}{f} \leq  \frac{t}{1+\epsilon_n} - \epsilon_n \right) \\ 
& \ +   P\left( \frac{Z^TZ}{f} \leq  \frac{t}{1+\epsilon_n}  - \epsilon_n \right)-  P\left( \frac{Z^TZ}{f} \leq t \right)  - P^*\left(G_{boot} \cap \mathcal{A}_n\right) = I^* + II^* + III^*
\end{align*}
where $G_{boot}$ satisfies $P(G_{boot}^c) = 0$ and for some $\epsilon_n \rightarrow 0$, is defined as:
\begin{align*}
G_{boot} = \{ |R_1^* - 1| \leq \epsilon_n, |R_2^*|,|R_3^*|, |R_4^*|  \leq \epsilon_n \ \}    
\end{align*}
For $I$, we may use Lemma~ \ref{lem:covdiff}, which establishes that bootstrap version of the covariance matrix, which consists of empirical covariances, is close to the Gaussian approximation, implying, by our Gaussian comparison result Lemma \ref{lem:gauss-compare}:
\begin{align*}
I^* = O_P\left( \left(\frac{\mathbb{E}[\norm{X_i X_i^T - \Sigma}^4] }{n (\lambda_1 - \lambda_2) \norm{M}_F^2}\right)^{1/4} \right)      
\end{align*}
For $II^*$, we may use the anti-concentration result and $P^*(G_{boot} \cap \mathcal{A}_n) \xrightarrow{P} 0$ by Markov's inequality since the Lemmas hold for the unconditional measure, which is the expectation of the bootstrap measure.  We may use analogous reasoning to the CLT for the upper bound and the result follows.  

\end{proof}
\subsection{Proof of Lemma~\ref{lem:covdiff}}\label{sec:lemcovdiff}
\begin{proof}
Let $Y_i:=X_iX_i^T-\Sigma$. Also let $M_i=\E[D_{i-1}Y_i v_1 v_1^T Y_i D_{i-1}]$.
First note that
\ba{\label{eq:varDiffDecomp}
\E^* ZZ^T -\bigvar&= \frac{\eta_n}{2n}\sum_i D_{i-1} (Y_i-Y_{i-1}) v_1 v_1^T (Y_i-Y_{i-1}) D_{i-1}\notag\\
&=\frac{\eta_n}{n}\sum_i\frac{\bb{D_{i-1} Y_i v_1 v_1^T Y_i D_{i-1}-M_i}+\bb{D_{i-1} Y_{i-1} v_1 v_1^T Y_{i-1} D_{i-1}-M_i}}{2}\notag\\
&+\frac{\eta_n}{n}\sum_i \bb{D_{i-1} Y_i v_1 v_1^T Y_{i-1} D_{i-1}+D_{i-1} Y_{i-1} v_1 v_1^T Y_{i} D_{i-1}}
}

We first compute trace.
\bas{
\trace(\E^* ZZ^T -\bigvar)&=\frac{\eta_n}{2n}\sum_i \underbrace{\bb{\|D_{i-1}Y_i v_1\|^2-\E \|D_{i-1}Y_i v_1\|^2}}_{U_{1,i}}\\
&+\frac{\eta_n}{2n}\sum_i \underbrace{\bb{\|D_{i-1}Y_{i-1} v_1\|^2-\E \|D_{i-1}Y_i\|^2}}_{U_{2,i}}\\
&+\frac{\eta_n}{n}\sum_i \underbrace{v_1 Y_i D_{2(i-1)}Y_{i-1} v_1}_{U_{3,i}}
}

The last step is true because $D_{i-1}^2=D_{2(i-1)}$.
We start with the first term.
\bas{
\E U_{i,1}^2&\leq \E\|D_{i-1}Y_i v_1\|^4\leq \E \|Y_i\|^4\left(\frac{1+\eta_n\lambda_2/n}{1+\eta_n\lambda_1/n}\right)^{4(i-1)}\\
\var(\sum_i U_{1,i})&\leq \E \|Y_1\|^4 \sum_i \left(\frac{1+\eta_n\lambda_2/n}{1+\eta_n\lambda_1/n}\right)^{4(i-1)}\\
&\leq \frac{n}{\eta_n(\lambda_1-\lambda_2)}\\
&\leq \frac{n}{\eta_n}\E \|Y_1\|^4\min\bb{\frac{1}{\lambda_1-\lambda_2},\eta_n}
}

Finally, 
\bas{
\E [U_{3,i}^2]\leq \E\bb{v_1 Y_i D_{2(i-1)}Y_{i-1} v_1}^2\leq M_d^2 \left(\frac{1+\eta_n\lambda_2/n}{1+\eta_n\lambda_1/n}\right)^{2(i-1)}
}
Thus, we have
\bas{
\frac{\eta_n}{2n}\sum_i
U_{1,i}=O_P\bb{\sqrt{\frac{ \E\|Y_1\|^4}{n(\lambda_1-\lambda_2)}}}
}
We also have,
\bas{
\frac{\eta_n}{2n}\sum_i
U_{2,i} = O_P\bb{\sqrt{\frac{ \E\|Y_1\|^4}{n(\lambda_1-\lambda_2)}}}
}

Also note that while $U_{3,i}$ terms are 1-dependent, they are in fact uncorrelated. 
Thus, we have:
\bas{
\var(\sum_i U_{3,i})\leq \frac{M_d^2 n}{(\lambda_1-\lambda_2)},
}
and,
\bas{
\trace(\E^*ZZ^T-\bigvar)=O_P\bb{\sqrt{\frac{\E\|X_iX_i^T-\Sigma\|^4}{n(\lambda_1-\lambda_2)}}}
}

Now we bound the Frobenius norm. We will start with the expected Frobenius norm of the first term of Eq~\ref{eq:varDiffDecomp}.
\bas{
A_1 &=\E\left\|\frac{1}{2n}\sum_{i=1}^nD_{i-1} Y_i v_1 v_1^T Y_i D_{i-1}-M_i\right\|_F^2\\
&\leq  \frac{1}{4 n^2}\sum_i \E\|D_{i-1}Y_i v_1v_1^T Y_iD_{i-1}\|_F^2\leq \frac{ E\|Y_1\|^4}{4n\eta_n(\lambda_1-\lambda_2)}
}
Similarly,
\bas{
A_2&=\E\left\|\frac{1}{n}\sum_i D_{i-1}Y_iv_1v_1^TY_{i-1}D_{i-1}\right\|_F^2\\
&\leq \frac{1}{n\eta_n(\lambda_1-\lambda_2)} M_d^2
}

Thus
,
\bas{
\left\| \E^* ZZ^T-\bigvar \right\|_F= O_P\bb{\sqrt{\frac{\E\|X_1X_1^T-\Sigma\|^4}{n(\lambda_1-\lambda_2)}}}
}

\end{proof}
\subsection{The Gaussian comparison lemma}\label{sec:gausscompare}
We use the following lemma to compare to Gaussian random variables with mean $0$ and different covariance matrices.  Our result is related to \cite{gotzeCOV}, but our lemma below is easier to implement and does not require that $3\|\Sigma \|^2 \leq \|\Sigma \|_{F}^2$.
\begin{lemma}\label{lem:gauss-compare}[Comparison lemma for inner products of Gaussian random variables]

Suppose that  $Z \sim N(0, \mathbb{V} )$, $\check{Z} \sim N(0, \check{\mathbb{V}})$, $f = \| \mathbb{V}\|_F$, and $\Delta_1 = \mathrm{tr}(\mathbb{V} - \check{\mathbb{V}})$.  Then, there exists some constant $K >0$ such that for any $\epsilon >0$, 
\begin{align*}
\sup_{t \in \mathbb{R}} \left| P( Z^T Z \leq  t) - P( \check{Z}^T \check{Z} \leq  t ) \right| \lesssim  \sqrt{\frac{|\Delta_1| + \epsilon}{f}} + \mathrm{exp}\left\{ - \left(\frac{\epsilon^2}{K^2\| \mathbb{V} -  \check{\mathbb{V}} \|_{F}^2} \bigwedge \frac{\epsilon}{ K\| \mathbb{V} -  \check{\mathbb{V}} \|} \right) \right\}
\end{align*}
\end{lemma}

\begin{proof}
Let $\lambda_1 \geq \ldots \geq \lambda_p$ denote the eigenvalues $\mathbb{V}$, $\gamma \geq \ldots \geq \gamma_p$ denote the eigenvalues of $\check{\mathbb{V}}$.   Recall that $Z^T Z \sim \sum_{r=1}^p \lambda_r \eta_r$, $\check{Z}^T \check{Z} \sim \sum_{r=1}^p \gamma_r \eta_r$, where $\eta_r \sim \chi^2(1)$.  We upper bound the difference between the CDFs uniformly in $t$; the argument for the lower bound is analogous.  For $\epsilon >0$, let $t' = t - |\Delta_1| -\epsilon$.  It follows that: 
\begin{align*}
& \ P( Z^T Z \leq  t) - P( \check{Z}^T \check{Z} \leq  t ) \\ 
 = & \ P\left(  \frac{ \sum_{r=1}^p \lambda_r \eta_r}{f}  \leq \frac{t}{f}  \right) - P\left( \frac{ \sum_{r=1}^p \lambda_r \eta_r + \sum_{r=1}^p (\gamma_r - \lambda_r) \eta_r - \Delta_1  }{f}  \leq \frac{t - \Delta_1}{f} \right) \\ 
\leq & \ P\left( \frac{t^\prime}{f} \leq \frac{ \sum_{r=1}^p \lambda_r \eta_r}{f}  \leq \frac{t^\prime + |\Delta_1| + \epsilon}{f}  \right) + P\left( \left|\sum_{r=1}^p (\gamma_r - \lambda_r) \eta_r - \Delta_1\right| > \epsilon \right)
\end{align*}
Observe that $\sum_{r=1}^p (\lambda_r - \gamma_r)^2 \leq \| \mathbb{V} -  \check{\mathbb{V}} \|_{F}^2$ by Hoffman-Wielandt inequality and $\max_r | \lambda_r - \gamma_r | \leq \|  \mathbb{V} -  \check{\mathbb{V}} \|_{op}$ by Weyl's inequality.  Since $\chi^2(1)$ is sub-Exponential, by Bernstein's inequality (see for example Theorem 2.8.2 of \cite{verhsynin-high-dimprob}:
\begin{align*}
P\left( \left|\sum_{r=1}^p (\gamma_r - \lambda_r) \eta_r - \Delta_1\right| > \epsilon \right) &\leq \mathrm{exp}\left\{ - \left(\frac{\epsilon^2}{K^2\| \mathbb{V} -  \check{\mathbb{V}} \|_{F}^2} \bigwedge \frac{\epsilon}{ K\| \mathbb{V} -  \check{\mathbb{V}} \|} \right)  \right\}
\end{align*}
The upper bound follows from an application of Proposition \ref{prop-chisq-ac}.  The lower bound is analogous.  

\end{proof}

\subsection{Other supporting lemmas for bootstrap consistency}
\label{sec:bootsupplem}
Before presenting our supporting lemmas, we present some events we will use frequently.
Let $\ac_{\sin}$ denote the set 
\begin{align}\label{eq:sinset}
    \ac_{\sin}:=\bbs{1-(v_1^T\vo)^2\leq \frac{\sinsqerr}{\sinsqprob}}.
\end{align} Using Corollary~\ref{cor:sinsq}, and the remark thereafter, we have:
\ba{\label{eq:sinset_prob}
P\bb{1-(v_1^T\vo)^2\geq \frac{\sinsqerr}{\sinsqprob}} \leq \sinsqprob,
}
where, under the assumptions of Theorem~\ref{thm-oja-clt},
\ba{\label{eq:sinsqerr}
\sinsqerr= C_3\frac{M_d\eta_n}{n(\lambda_1-\lambda_2)}
}

Also let,
\begin{align}\label{eq:highprobset}
    \mathcal{A}_
n =  \left\{\max_{1 \leq i \leq n } \|X_i\|_2^2 \leq \alpha_n \right\}
\end{align}
\begin{lemma}\label{lem:boot-res}[Bounding the norm of bootstrap residual from $T_1^*$]
Let $\Delta_i=X_i X_i^T-X_{i-1}X_{i-1}^T$ and assume the conditions in Theorem~\ref{thm-oja-clt}. 
Let $D_i=\vorth\Lambda_\perp^i\vorth^T$, where $\Lambda_\perp(k,\ell)=\frac{1+\eta_n\lambda_{k+1}/n}{1+\eta_n\lambda_1/n}1(k=\ell)$. For any $\epsilon,\delta>0$, we have:
\bas{
&P\left(\bbs{\sqrt{\frac{n}{\eta_n}}\left\|\frac{\vp\vp^T T_1^* v_1(v_1^T u_0)}{|v_1^T u_0|(1+\eta_n\lambda_1/n)^n} - \frac{\eta_n}{n}\sum_i W_i D_{i-1}\Delta_i v_1\right\| \geq \epsilon}\cap \ac_n\right)\\
&\leq C''\frac{\alpha_n M_d\eta_n^3\log d}{n\epsilon^2\delta}+\delta
}
\end{lemma}

\begin{proof}

\bas{
&\frac{\vp\vp^T T_1^*v_1 (v_1^Tu_0
)}{|v_1^Tu_0|(1+\eta_n\lambda_1/n)^{n-1}} \\
&= \sn(v_1^Tu_0)\frac{\eta_n}{n} \sum_i W_i D_{i-1}\Delta_iv_1\\ &+\sn(v_1^Tu_0)\frac{\eta_n}{n}\underbrace{(\vp\vp^T-\vorth\vorth^T)\sum_i W_i  D_{i-1}\Delta_i v_1}_{r_1}\\
&+ \sn(v_1^Tu_0)\frac{\eta_n}{n}\left(\underbrace{\sum_i W_i\bb{ \frac{R_{1,i-1}\Delta_i v_1}{\bb{1+\lambda_1\eta_n/n}^{i}}}}_{r_2}\right.\\
&\qquad\qquad\qquad\qquad\left.+\underbrace{\frac{W_i(I+\eta_n\lambda_1/n)^{i-1}\Delta_iR_{i,n} v_1}{\bb{1+\lambda_1\eta_n/n}^{n-1}}}_{r_3} + \underbrace{W_i\frac{R_{1,i-1}\Delta_i R_{i,n}v_1}{\bb{1+\lambda_1\eta_n/n}^{n-1}}}_{r_4} \right)
}

Define
\begin{align}
\mathcal{B}_{1,j} &= \prod_{i=1}^{j} \left(I + \frac{\eta_n}{n} X_i X_i^T   \right) \qquad \mathcal{B}_{j,n} = \prod_{i=j}^{n} \left(I + \frac{\eta_n}{n} X_i X_i^T \right) \label{eq:b1j}
\end{align} 
 
When $j=0$, $\mathcal{B}_{1,j}=I$.

Using Lemma~\ref{lem:ward}
we have:
\ba{\label{eq:wardbounds}
R_{1,i}&=\mathcal{B}_{1,i}-(I+\eta_n\Sigma/n)^i\qquad R_{i,n}=\mathcal{
B}_{i,n}-(I+\eta_n\Sigma/n)^{n-i}\\
\E\|R_{1,i-1}\|^2 &\leq   eM_d(1+2\log d)\frac{\eta_n^2}{n^2} i  \left(1+\eta_n\lambda_1/n\right)^{2i}\\
\E\|R_{i,n}\|^2 &\leq   eM_d(1+2\log d)\frac{\eta_n^2}{n^2} (n-i)  \left(1+\eta_n\lambda_1/n\right)^{2(n-i)}
}

We have, on the good set $\ac_{\sin}$,
\bas{
\E^*\|r_1\|^2 \leq n\alpha_n\frac{\sinsqerr}{\sinsqprob} 
}
We also have:
\bas{
\E\left[\E^*\|r_2\|^2 1(\ac_n)\right]&\leq  \frac{\eta_n^2}{
n^2}\alpha_n\sum_i \E[\|R_{1,i}^2 1(\ac_n)\|^2]\\
&\leq eM_d(1+2\log d)\alpha_n\eta_n^2
}
The last step is true because $\E[\|R_{1,i}^2 1(\ac_n)\|^2]\leq \E[\|R_{1,i}^2\|^2]$.
Similarly 
\bas{
\E\left[\E^*\|r_3\|^2 1(\ac_n)\right]\leq eM_d(1+2\log d)\alpha_n\eta_n^2
}
and 

\bas{
\E\left[\E^*\|r_4\|^2 1(\ac_n)\right]\leq e^2M_d^2(1+2\log d)^2\alpha_n\eta_n^4/n
}

Finally, we have:
\bas{
&P\left(\bbs{\frac{\eta_n}{n}\|\sum_j r_j\|^2\geq \epsilon}\cap\ac_n\right)\leq P\left(\bbs{4\frac{\eta_n}{n}\sum_j \|r_j\|^2\geq \epsilon}\cap \ac_n\right)\\
&\leq \sum_i P\bb{\bbs{\|r_i\|^2\geq \frac{n\epsilon}{16\eta_n}}\cap \ac_n\cap \ac_{\sin}}+\sinsqprob\\
&\leq C\sum_i \E\left[\E^*\|r_i\|^21(\ac_n\cap \ac_{\sin})\right]\times\frac{\eta_n}{n\epsilon}+\sinsqprob\\
&\stackrel{(i)}{\leq} C'\bb{n\alpha_n\frac{\sinsqerr}{\sinsqprob}+M_d\log d\alpha_n\eta_n^2}\times\frac{\eta_n}{n\epsilon}+\sinsqprob\\
&\stackrel{(ii)}{\leq} C''\frac{\alpha_n M_d\eta_n^3\log d}{n\epsilon\sinsqprob}+\sinsqprob
}
Step (i) is true because $M_d\log d\eta_n^2/n\rightarrow 0$.
Step (ii) is true because of Eq~\ref{eq:sinsqerr}. Now setting $\sinsqprob$ to any $\delta>0$ gives the result.

\end{proof}

\begin{lemma}[Concentration of the norm for the bootstrap]
\label{lem:norm-concentration-lemma-boot}
Let $u_0$ be uniformly distributed on $\mathbb{S}^{d-1}$ and $a_1 = u_0^\prime v_1$  and $V_\perp V_\perp^T$ is orthogonal complement.
Suppose that $(\alpha_n)_{n \geq 1}$ satisfies $0 \leq \frac{(\eta_n \alpha_n)^2}{n} \leq 1$.
Then, for any $\epsilon >0, 0< \delta <1$ and some $C > 0$, 
\begin{align*}
& \ P\bb{\left\{\left|\frac{\|B_n^*u_0\|}{|a_1|(1+\eta_n\lambda_1/n)^n}-1\right|\geq \epsilon_n \right\} \  \bigcap \ \mathcal{A}_{n} } \\ 
\leq &  \ 
\frac{ d\exp\left(-\eta_n(\lambda_1 - \lambda_2) + \frac{\eta_n^2}{n} (\lambda_1^2 + M_d) \right) + \frac{\eta_n^2}{n} M_d \exp\left( \frac{\eta_n^2}{n}  \right) }{8\log^{-1}(1/\delta) \delta^2 \  \epsilon^2\left(1+\frac{\eta_n^2 \lambda_1^2}{n} \right) } \\ 
& \ + \frac{e^2 \eta_n^2 M_d(1+ \log d) }{2n\epsilon^2} + \frac{C \beta_n^* \log(1/\delta) }{(1- \beta_n^*) \delta^{2} \epsilon^2 } + C \delta,
\end{align*}
where $\beta_n^*$ is defined in (\ref{eq:b-star}) and $\ac_n$ is defined in Eq~\ref{eq:highprobset}.

\end{lemma}
\begin{proof}
First note that we may reduce the problem as follows:

 \begin{align*}
     & \ P \bb{\bbs{\left|\frac{\|B_n^*u_0\|}{|a_1|(1+\eta_n\lambda_1/n)^n}-1 \right| \geq \epsilon}\cap\ac_n }\\
     \leq & \ P\bb{\bbs{ \frac{\norm{B_n^*u_0 -B_nu_0}_2} {|a_1|(1+\eta_n\lambda_1/n)^n } + \left|\frac{\norm{B_nu_0}_2}{|a_1|(1+\eta_n\lambda_1/n)^n} -1 \right| > \epsilon}\cap \ac_n}   
     \\ \leq & \  \mathbb{E} \left[P^*\bb{ \frac{\norm{B_n^*u_0 -B_nu_0}_2} {|a_1|(1+\eta_n\lambda_1/n)^n } > \frac{\epsilon}{2}} 1(\ac_n)\right] + P\bb{ \left|\frac{\norm{B_nu_0}_2}{|a_1|(1+\eta_n\lambda_1/n)^n} -1 \right| > \frac{\epsilon}{2}} 
    %
     %
 \end{align*}
 The bound for the second term follows from Lemma \ref{lemma-norm-concentration}. For the first term, we invoke Proposition~\ref{prop:trace} so that, with probability at least $1- C \delta$,
\begin{align*}
 & \  \frac{\norm{(B_n^* -B_n)g}_2^2} { (v_1^T g)^2 (1+\eta_n\lambda_1/n)^{2n}}
 \leq  \   \frac{ \log(1/\delta) \ \|B_n^*-B_n\|_F^2}{\delta^{2}(1+\eta_n\lambda_1/n)^{2n}} 
 \end{align*}
 
 
 Now, using the fact that for any two Hayek projections $P_S^*$ and $P_T^*$, $\mathbb{E}[(P_S^*)^TP_T^*]$ = 0 and for any two matrices $\|AB\|_F \leq \|A\|_{F} \|B\|_{op}$, we have on the high probability set:
 \begin{align*}
 & \ \ \E^*\|B_n^*-B_n\|_F^2 \\
 \leq & \  \sum_{k=1}^n \sum_{|S| = k} \left(\frac{\eta_n}{n}\right)^{2k} \prod_{i=1}^k  \norm{X_{S[i]}X_{S[i]}^\prime - X_{S[i]-1}X_{S[i]-1}^\prime}_{F}^2 \ \prod_{j=1}^{k+1} \norm{\mathcal{B}_{j,n}^{(S)}}_{op}^2,
 \end{align*}
 where $\mathcal{B}_{j,n}^{(S)}$ denotes a contiguous block of  $I + \frac{\eta_n}{n} X_iX_i^T$ only. More precisely, suppose $|S| = k$.  Let $S[i]$ denote the $i$th element of $S$, with $S[0]= 0$ and $S[k+1] =n-1$.  For each $1 \leq j \leq k+1$ if $S[j] > S[j-1]+1$ define $\mathcal{B}_{j,n}$ as:
\begin{align}\label{eq:bjn}
\mathcal{B}_{j,n}^{(S)} &= \prod_{i=S[j-1]+1}^{S[j]-1} \left(I + \frac{\eta_n}{n} X_i X_i^T \right) 
\end{align}   
otherwise, set $\mathcal{B}_{j,n}^{(S)} =I$.
 Now, we may repeat arguments in Lemma \ref{lemma-boot-hayek-remainder} equations (\ref{eq:sub-mult}), (\ref{eq:hayek-boot-summand}), and (\ref{eq:hayek-boot-upper-bound}) to conclude that, for some $C > 0$,

 \begin{align*}
 P\left( \frac{\log(1/\delta) \ \|B_n^*-B_n\|_F^2}{ \delta^2 (1+\eta_n\lambda_1/n)^{2n}} > \epsilon \  \bigcap \  \mathcal{A}_n \right) \leq  \frac{C \log(1/\delta) \beta_n^*}{(1- \beta_n^*) \delta^{2} \epsilon^2 }  
 \end{align*}
The result follows.  
\end{proof}

\begin{lemma}[Negligibility of higher-order Hoeffding projections for the bootstrap]
\label{lemma-boot-hayek-remainder}
Suppose $\alpha_n$ is defined so that $0 \leq \beta_n^* \leq 1$, where
\begin{align}
\label{eq:b-star}
\beta_n^* = \exp\left(\sqrt{\frac{C M_d^2 \eta_n^2 \log d}{n}} \right) \frac{4 \eta_n^2   \alpha_n^2}{n} \ 
\end{align}

Then for any $\epsilon >0 , 0 < \delta < 1$ and for some $C>0$,  
\begin{align*}
&P\left(\left\{ \frac{\sqrt{\frac{n}{\eta_n}}  \norm{\hat{V}_{\perp} \hat{V}_{\perp}^T \sum_{k >1} T_k^*u_0}}{|a_1|(1+ \frac{\eta_n\lambda_1}{n})^{n}} > \epsilon_n \right\}   \  \bigcap \ \mathcal{A}_n \right) \\
&\leq  \exp\left(\sqrt{\frac{C M_d^2 \eta_n^2 \log d}{n}} \right)    \frac{\log(1/\delta)}{\delta^2} \frac{\alpha_n^2 \beta_n^* \eta_n}{(1-\beta_n^*) \epsilon^2 } + C \delta    ,
\end{align*}
where $\ac_n$ is defined in Eq~\ref{eq:highprobset}.
\end{lemma}
\begin{proof}


Using the trace trick in Proposition~\ref{prop:trace} again, we have that, with probability at least $1-C\delta$ for some $C >0$,
\begin{align*}
& \frac{\frac{n}{\eta_n}  \norm{\hat{V}_{\perp} \hat{V}_{\perp}^T \sum_{k >1} T_k^*g}^2}{(v_1^Tg)^2(1+ \frac{\eta_n\lambda_1}{n})^{2n}} 
\leq \frac{\frac{n}{\eta_n} \log(1/\delta) \norm{\sum_{k >1} T_k}_F^2 }{\delta^2 (1+ \frac{\eta_n\lambda_1}{n})^{2n} }
\end{align*}

The Hoeffding decomposition (Proposition \ref{prop:boothoeff}), together with the fact that $\norm{AB}_F \leq \norm{A}_F \norm{B}_{op}$ implies: 
\begin{align}
\label{eq:sub-mult}
\begin{split}
& \ \mathbb{E}^*\left[\norm{\sum_{k >1} T_k^* }_F^2 \right] = \mathbb{E}^*\left[  \sum_{k >1} \norm{T_k^*}_F^2 \right]
\\ \leq & \ \sum_{k=2}^n \sum_{|S| = k} \left(\frac{\eta_n}{n}\right)^{2k} \prod_{i=1}^k  \norm{X_{S[i]}X_{S[i]}^T - X_{S[i]-1}X_{S[i]-1}^T}_{F}^2 \ \prod_{j=1}^{k+1} \norm{\mathcal{B}_{j,n}^{(S)}}_{op}^2
\end{split}
\end{align}


Now, that expectation corresponding to a given summand is given by: 
\begin{align}
\label{eq:hayek-boot-summand}
\begin{split}
\ & \  \int_{\mathcal{A}_n } \norm{X_{S[i]}X_{S[i]}^T - X_{S[i]-1}X_{S[i]-1}^T}_{F}^2 \ \prod_{j=1}^{k+1} \norm{\mathcal{B}_{j,n
}^{(S)}}^2 dP \\  
 \leq & \int_{\mathcal{A}_n } \prod_{i=1}^k 4 \alpha_n^2  \ \prod_{j=1}^{k+1} \norm{\mathcal{B}_{j,n}^{(S)}}^2 dP \\
  \leq  & \ \left( 4 \alpha_n^2 \right)^k \ \prod_{j=1}^{k+1} \mathbb{E}\left[\norm{\mathcal{B}_{j,n}^{(S)}}^2\right] 
  \end{split}
\end{align}
where $\mathcal{B}_{j,n}^{(S)}$ is defined in Eq~\ref{eq:bjn}.

To bound $\mathbb{E}\left[\norm{\mathcal{B}_{j,n}^{(S)}}^2\right]$, we invoke Lemma~\ref{lem:ward} Eq~\ref{eq:ward-op-exp}.
For some $C> 0$ uniformly in $S$: 
\begin{align*}
\prod_{j=1}^{k+1} \mathbb{E}\left[\norm{\mathcal{B}_{j,n}^{(S)}}^2\right] \leq \exp\left(\sqrt{\frac{C M_d^2 \eta_n^2 \log d}{n}} \right)^{k+1} \left(1 + \frac{\eta_n \lambda_1}{n}\right)^{2(n-k)}
\end{align*}
Therefore, by Markov's inequality,
\begin{align}
\begin{split}
\label{eq:hayek-boot-upper-bound}
& \ P\left(\left\{ \frac{\sqrt{\frac{n}{\eta_n}}  \norm{\hat{V}_{\perp} \hat{V}_{\perp}^T \sum_{k >1} T_k^*u_0}}{(1+ \frac{\eta_n\lambda_1}{n})^{n}} > \epsilon_n \right\}   \  \bigcap \ \mathcal{A}_n \right) \\ 
\leq & \ \frac{n}{\delta^{3}\epsilon_n^2\eta_n} \exp\left(\sqrt{\frac{C M_d^2 \eta_n^2 \log d}{n}} \right) \sum_{k=2}^n  \left(\frac{4 \eta_n^2   \alpha_n^2}{n} \ \exp\left(\sqrt{\frac{C M_d^2 \eta_n^2 \log d}{n}} \right) \right)^{k}
\\ \leq
 & \ \alpha_n^2 \eta_n \delta_n^{-3}  \epsilon_n^{-2} \exp\left(\sqrt{\frac{C M_d^2 \eta_n^2 \log d}{n}} \right)  \sum_{k=1}^n  \left(\frac{4 \eta_n^2   \alpha_n^2}{n} \ \exp\left(\sqrt{\frac{C M_d^2 \eta_n^2 \log d}{n}} \right) \right)^{k} \\ 
  \leq & \ \exp\left(\sqrt{\frac{C M_d^2 \eta_n^2 \log d}{n}} \right)    \frac{\alpha_n^2 \beta_n^* \eta_n}{(1-\beta_n^*) \epsilon_n^2\delta_n^{3}}
\end{split}
\end{align}
where the last line follows from a geometric series argument.
\end{proof}

\begin{lemma}\label{lem:geometric}
$$\sum_{i=0}^n \bb{1-\frac{\eta_n/n(\lambda_1-\lambda_2)}{1+\eta_n\lambda_1/n}}^{2i}\leq \frac{n}{\eta_n}\min\bb{\eta_n,\frac{1}{\lambda_1-\lambda_2}}$$
\end{lemma}
\begin{proof}
This follows from the definition of a geometric series.
\end{proof}

\begin{lemma}[Bounding the leading Hoeffding projection for the bootstrap on $V_\perp$]\label{lem:hajekremboot}
 Let $\lambda_1M_d(\log d)^2\frac{\eta_n^2}{n}\rightarrow 0$, and $\condition\rightarrow 0$. For any $\epsilon,\delta>0$, and $C_1,C_2\geq 0$, we have:
\bas{
P\bb{\bbs{\sqrt{\frac{n}{\eta_n}}\frac{\|\vp\vp^T T_1^* \vorth\vorth^T u_0\|}{(1+\eta_n\lambda_1/n)^{n}|v_1^T u_0|}\geq \epsilon}\cap \ac_n}\leq \frac{C_1\alpha_nM_d\eta_n^2\log(1/\delta)}{n(\lambda_1-\lambda_2)\delta^3}\frac{1}{\epsilon^2}+C_2\delta 
}
\bk
\end{lemma}
\begin{proof}
Using Proposition~\ref{prop:trace}, with probability at least $1-\delta$, 
\ba{\label{eq:normalized_higher}
\frac{\|\vp\vp^T T_1^* \vorth\vorth^T u_0\|^2}{(1+\eta_n\lambda_1/n)^{2n}\|v_1^T u_0\|^2}&\leq \frac{\log(1/\delta) \norm{\vp\vp^T T_1^* \vorth\vorth^T}_F^2 }{\delta^2(1+\eta_n\lambda_1/n)^{2n} }\notag\\ &=\frac{\log(1/\delta)\trace(\vp\vp^TT_1^*\vorth\vorth^T T_1^* \vp\vp^T)}{\delta^2 (1+\eta_n\lambda_1/n)^{2n}}\notag\\
&=\frac{\log(1/\delta)\norm{\vp\vp^TT_1^*\vorth}_F^2}{\delta^2 (1+\eta_n\lambda_1/n)^{2n}}
}


First note that, 
\bas{
\|V_\perp V_\perp^T-\vp\vp^T\|_F^2=\|v_1v_1^T-\vo\vo^T\|_F^2=2(1-(v_1^T\vo)^2)
}
Thus, we have
 \ba{\label{eq:T1v1decomp}
 &\E^*\|\vp\vp^T T_1^* V_\perp \|_F^2\notag\\
 &=\frac{\eta_n^2}{n^2}\sum_i \|\vp\vp^T \mathcal{B}_{1,i-1}(X_iX_i^T-X_{i-1}X_{i-1}^T)\mathcal{B}_{i+1,n}  V_\perp \|_F^2\notag\\
 &\leq 4\frac{\eta_n^2}{n^2}\sum_i\sum_{j=1}^6 \|r_{j,i}\|_F^2,
 }
 where $B_{1,i}$ are defined in Eq~\ref{eq:wardbounds}, and the residual vectors $r_{k,i}$ are defined as follows.
Recall the definition of $R_{1,i}$ and $R_{i,n}$ from Eq~\ref{eq:wardbounds}.
Now define the following vectors which contribute to the remainder.
\bas{
r_{1,i}&=\vp\vp^T R_{1,i-1}(Y_i-Y_{i-1})R_{i+1,n} V_\perp \\
r_{2,i}&=\vp\vp^T R_{1,i-1}(Y_i-Y_{i-1})(I+\eta_n/n\Sigma)^{n-i} V_\perp\\
r_{3,i}&=\vorth\vorth^T (I+\eta_n/n\Sigma)^{n-i}(Y_i-Y_{i-1})R_{i+1,n} V_\perp \\
r_{4,i}&=\vorth\vorth^T (I+\eta_n/n\Sigma)^{n-i}(Y_i-Y_{i-1})(I+\eta_n/n\Sigma)^{n-i} V_\perp \\
r_{5,i}&=(\vp\vp^T-\vorth\vorth^T) (I+\eta_n/n\Sigma)^{n-i}(Y_i-Y_{i-1})R_{i+1,n} V_\perp\\
r_{6,i}&=(\vp\vp^T-\vorth\vorth^T) (I+\eta_n/n\Sigma)^{n-i}(Y_i-Y_{i-1})(I+\eta_n/n\Sigma)^{n-i} V_\perp\\
}

First we will bound $\|r_{1,i}\|_F^2$. Recall the set $\ac_n$ where the maximum norm is bounded from~\ref{eq:highprobset}.


\ba{\label{eq:r1}
E_{1,i}&:=\int_{\mathcal{A}_n}\|r_{1,i}\|_F^2 dP\leq 2\alpha_n\int_{\mathcal{A}_n}\notag \|R_{1,i-1}\|^2\|R_{i+1,n}\|^2 dP\notag\\
&\leq 2\alpha_n\int \|R_{1,i}\|^2\|R_{i+1,n}\|^2 dP\leq 2\alpha_n\E\|R_{1,i}\|^2\E\|R_{i+1,n}\|^2
}
Similarly, 
\ba{\label{eq:r2}
E_{2,i}&=\int_{\mathcal{A}_n}\|r_{2,i}\|_F^2 dP\leq 2\alpha_n\left(1+\eta_n\lambda_2/n\right)^{2(n-i)}\E\|R_{1,i-1}\|^2
}
\ba{\label{eq:r3}
E_{3,i}&=\int_{\mathcal{A}_n}\|r_{3,i}\|_F^2 dP\leq 2\alpha_n\left(1+\eta_n\lambda_2/n\right)^{2(i-1)}\E\|R_{i+1,n}\|^2
}
Similarly,
\ba{\label{eq:r4}
E_{4,i}=\int_{\mathcal{A}_n}\|r_{4,i}\|_F^2 dP &\leq 2\alpha_n \left(1+\eta_n\lambda_2/n\right)^{2(n-1)}
}

Recall the set $\ac_{\sin}$ from Eq~\ref{eq:sinset}. With probability at least $1-\sinsqprob$,
\bas{
E_{5,i}&=\int_{\mathcal{A}_n\cap \ac_{\sin}}\|r_{5,i}\|_F^2 dP\leq 4\alpha_n\frac{\sinsqerr}{\sinsqprob}\left(1+\eta_n\lambda_1/n\right)^{2(i-1)}\E\|R_{i+1,n}\|^2\\
E_{6,i}&=\int_{\mathcal{A}_n\cap \ac_{\sin}}\|r_{6,i}\|_F^2 dP\leq 2\alpha_n\frac{\sinsqerr}{\sinsqprob}\left(1+\eta_n\lambda_1/n\right)^{2(i-1)}\left(1+\eta_n\lambda_2/n\right)^{2(n-i)}
}

Observe that, using Eq~\ref{eq:wardbounds}, we have, 
\bas{
\ce_1&:=\sum_i E_{1,i}\leq \frac{2\alpha_n e^2M_d^2(1+2\log d)^2\eta_n^4}{n} (1+\eta_n\lambda_1/n)^{2(n-1)}\\
\ce_2&:=\sum_i (E_{2,i}+E_{3,i})\leq \frac{4\alpha_n e M_d(1+2\log d)\eta_n^3}{n}\min\bb{\eta_n,\frac{1}{\lambda_1-\lambda_2}} (1+\eta_n\lambda_1/n)^{2n-1}\\
\ce_3&:=\sum_i E_{
4,i}\leq 2\alpha_n n \left(1+\eta_n\lambda_2/n\right)^{2n}\\ 
}
With probability at least $1-\sinsqprob$,
we have
\bas{
 \ce_4:=\sum_i E_{5,i}&\leq 4\alpha_n\frac{\sinsqerr}{\sinsqprob}eM_d(1+2\log d)\eta_n^2(1+\eta_n\lambda_1)^{2(n-1)} \\
\ce_5:=\sum_i E_{6,i}&\leq 2\alpha_n\frac{\sinsqerr}{\sinsqprob}\frac{n}{\eta_n}\min\bb{\eta_n,\frac{1}{\lambda_1-\lambda_2}}(1+\eta_n\lambda_1)^{2(n-1)}
}
If $\lambda_1M_d(\log d)^2\frac{\eta_n^2}{n}\rightarrow 0$, then $\ce_4\leq C_1\ce_5$ for some positive constant $C_1$. If $nd\exp(-2\eta_n(\lambda_1-\lambda_2))\rightarrow 0$, then $\ce_3\leq C_2 \ce_5$. 

Thus, under these conditions, 
\bas{
\ce_1,\ce_2\leq C_4 \ce_5
}

With probability at least $1-\sinsqprob$, for some positive constant $C'$,
\bas{
\frac{\sum_{i=1}^5\ce_i}{(1+\eta_n\lambda_1/n)^{2n}}\leq C'\alpha_n\frac{\sinsqerr}{\sinsqprob}
}

Finally, using Eq~\ref{eq:T1v1decomp} we get:
\ba{\label{eq:allbound}
 \frac{\int_{\ac_{\sin}\cap \ac_n}\E^*\|\vp\vp^T T_1^* V_\perp \|_F^2 dP}{(1+\eta_n\lambda_1/n)^{2n}}
 &\leq C''\alpha_n\frac{\eta_n^2}{n}\frac{\sinsqerr}{\sinsqprob}
 }

 Let $\ac_1$ denote the set where Eq~\ref{eq:normalized_higher} holds.  
\bas{
&P\bb{\bbs{\frac{n}{\eta_n}\frac{\|\vp\vp^T T_1^* \vorth\vorth^T u_0\|^2}{(1+\eta_n\lambda_1/n)^{2n}(v_1^T u_0)^2}\geq \epsilon}\cap \ac_n}\\
&\leq P\bb{\bbs{\frac{ \norm{\vp\vp^T T_1^* \vorth}_F^2 }{(1+\eta_n\lambda_1/n)^{2n} } \geq \frac{\epsilon \delta^2}{\log (1/\delta)}\frac{\eta_n}{n}}\cap \ac_n\cap \ac_1}+2\delta\\
&\leq P\bb{\bbs{\frac{ \norm{\vp\vp^T T_1^* \vorth}_F^2 }{(1+\eta_n\lambda_1/n)^{2n} } \geq \frac{\epsilon \delta^2}{\log (1/\delta)}\frac{\eta_n}{n}}\cap \ac_n\cap \ac_1\cap\ac_{\sin}}+2\delta+\sinsqprob\\
&\leq \E\left[\frac{\E^* \norm{\vp\vp^T T_1^* \vorth}_F^2 }{(1+\eta_n\lambda_1/n)^{2n}}\times \frac{\log(1/\delta)n}{\epsilon\delta^2\eta_n}1(\ac_n\cap \ac_1\cap\ac_{\sin})\right]+2\delta+\sinsqprob\\
&\stackrel{(i)}{\leq}
 \frac{C''\alpha_n\eta_n\log(1/\delta)}{\sinsqprob\delta^2}\frac{\sinsqerr}{\epsilon} +2\delta+\sinsqprob\\
 &\stackrel{(ii)}{\leq} \frac{C'''\alpha_nM_d\eta_n^2\log(1/\delta)}{n(\lambda_1-\lambda_2)\sinsqprob\delta^2}\frac{1}{\epsilon} +2\delta+\sinsqprob 
}
Step (i) follows from Eq~\ref{eq:allbound}. Step (ii) follows from the definition of $\sinsqerr$ in Eq~\ref{eq:sinsqerr}. Now setting $\sinsqerr=\delta$, we get the result.
\bk
\end{proof}

\section{Proof of  Proposition~\ref{prop:vardecay}}\label{sec:D}
\begin{proof}[Proof of Proposition~\ref{prop:vardecay}]
Since $\norm{X_{1j}}_{\psi_2} \leq \nu_j$ it follows that  $\norm{X_{1j}^2}_{\psi_1} \leq \nu_j^2$.  Observe that $(X_{1j}^2 - \mathbb{E}X_{1j}^2)/\nu_j^2$ is sub-Exponential with parameter at most 1 since $\norm{(X_{1j}^2 - \mathbb{E}[X_{1j}^2])/\nu_j^2 }_{\psi_1} \leq \norm{X_{1j}^2}_{\psi_1}/\nu_j^2 = 1$. 
By multivariate Holder inequality with $p_j = \sum_{j=1}^d\nu_j^2 / \nu_j^2$ and property (e) of Proposition 2.7.1 of \cite{verhsynin-high-dimprob}, for $|\lambda| < 1/(\sum_{i=1}^d \nu_i^2)$:  
\begin{align*}
 \ \mathbb{E}\left[\exp\left(\lambda\sum_{j=1}^d (X_{1j}^2 - \mathbb{E}[X_{1j}^2])  \right) \right]
& \leq  \prod_{j=1}^d \mathbb{E}\left[\exp\left(\lambda (X_{1j}^2  - \mathbb{E}[X_{1j}^2])  \right)^{\frac{\sum_{i=1}^d \nu_i^2}{\nu_j^2}} \right]^{\frac{\nu_j^2}{\sum_{i=1}^d \nu_i^2 }} \\ 
 & =  \prod_{j=1}^d \mathbb{E}\left[\exp\left( \frac{ \lambda (\sum_{i=1}^d \nu_i^2) (X_{1j}^2-\mathbb{E}[X_{1j}^2]) }{\nu_j^2} \right) \right]^{\frac{\nu_j^2}{\sum_{i=1}^d \nu_i^2}} 
\\ & \leq \ \prod_{j=1}^d \exp\left( \frac{ K\lambda^2  (\sum_{i=1}^d \nu_i^2)^2 \nu_j^2}{\sum_{i=1}^d \nu_i^2 } \right)
\\ & = \exp\left\{K \lambda^2 \left(\sum_{i=1}^d \nu_i^2\right)^2   \right\}
\end{align*}
Therefore, $\norm{\sum_{i=1}^d X_{1i}^2}_{\psi_1} \leq \sum_{i=1}^d \nu_i^2$.  Since a subexponential random variable  $T$ satisfy the tail condition:
\begin{align*}
P(T - \mathbb{E}[T]  > t) \leq \exp(-t/K \nu )     
\end{align*}
for another universal constant $K>0$, the second claim follows by a union bound and noting that $\mathbb{E}[\norm{X_1}_2^2] \leq \sum_{i=1}^d \nu_i^2 < C_2$ since absolute summability implies square summability.    
\end{proof}

\pagebreak

\bibliographystyle{abbrv}
\bibliography{references}

 \makeatletter\@input{zz.tex}\makeatother